\RequirePackage{fix-cm}
\documentclass[smallextended,envcountsect,a4paper,11pt]{svjour3}
\usepackage{fullpage}
\smartqed  % flush right qed marks, e.g. at end of proof
\usepackage[active]{srcltx}
\usepackage{latexsym,amssymb}
\usepackage{amsmath}
\usepackage{stmaryrd}
\usepackage[T1]{fontenc}
\usepackage[latin1]{inputenc}
\usepackage{hyperref}
\usepackage{enumerate}
\usepackage{pgf,tikz}
\usetikzlibrary{shapes,arrows,positioning,calc}
\newcommand{\op}{\llbracket}                                   
\newcommand{\cl}{\rrbracket}
\def\cont{\mathsf c}
\def\pv#1{\ensuremath{\mathsf{#1}}}

\def\Om#1#2{\ensuremath{\overline{\Omega}_{#1}{\pv{#2}}}}
\def\Cl#1{\ensuremath{\mathcal{#1}}}

\newtheorem{MyThm}{Theorem}
\newtheorem{Thm}[theorem]{Theorem}
\newtheorem{Prop}[theorem]{Proposition}
\newtheorem{Lemma}[theorem]{Lemma}
\newtheorem{Cor}[theorem]{Corollary}

\begin{document}

\title{Pseudovarieties of ordered completely regular semigroups%
  \thanks{The first author acknowledges partial funding by CMUP
    (UID/MAT/00144/2013) which is funded by FCT (Portugal) with
    national (MCTES) and European structural funds (FEDER) under the
    partnership agreement PT2020. The work was carried out in part at
    Masaryk University, whose hospitality is gratefully acknowledged,
    with the support of the FCT sabbatical scholarship
    SFRH/BSAB/142872/2018. %
    The second author was supported by the Institute for Theoretical
    Computer Science (GAP202/12/G061), Czech Science Foundation. }}

\author{Jorge Almeida
  \and
        Ond\v rej Kl\'\i ma
}

\authorrunning{J. Almeida and O. Kl\'ima}

\institute{J. Almeida\at
  CMUP, Dep.\ Matem\'atica, Faculdade de Ci\^encias,
  Universidade do Porto, Rua do Campo Alegre 687, 4169-007 Porto,
  Portugal\\
  \email{jalmeida@fc.up.pt} \\
  \and
  O. Kl\'\i ma\at
  Dept.\ of Mathematics and Statistics, Masaryk University,
  Kotl\'a\v rsk\'a 2, 611 37 Brno, Czech Republic\\
  \email{klima@math.muni.cz}
}

\date{}

\maketitle

\begin{abstract}
  This paper is a contribution to the theory of finite semigroups and
  their classification in pseudovarieties, which is motivated by its
  connections with computer science. The question addressed is what
  role can play the consideration of an order compatible with the
  semigroup operation. In the case of unions of groups, so-called
  completely regular semigroups, the problem of which new
  pseudovarieties appear in the ordered context is solved. As
  applications, it is shown that the lattice of pseudovarieties of
  ordered completely regular semigroups is modular and that taking the
  intersection with the pseudovariety of bands defines a complete
  endomorphism of the lattice of all pseudovarieties of ordered
  semigroups.
\end{abstract}

  

\section{Introduction}
\label{intro}

There are several fronts in which semigroup theory has developed as
semigroups appear naturally in many contexts. On the algebraic front,
varieties of semigroups, in the sense of universal algebra, have
received considerable attention in the literature. Thus, the study of
the lattice of all varieties of semigroups is one of the classical
topics in semigroup theory. One of the first results in this direction
was a complete description of the lattice of all varieties of bands
(semigroups in which all elements are idempotent) which was described
independently by Birjukov~\cite{Birjukov:1970},
Fennemore~\cite{Fennemore:1971a,Fennemore:1971aII} and
Gerhard~\cite{Gerhard:1970}. Bringing groups into play, a natural
generalization of both classes is the class of completely regular
semigroups (meaning semigroups in which every element lies in a
subgroup) which also forms a variety in a suitable algebraic
signature, that of unary semigroups. Consequently, intensive attention
was also paid to the study of the lattice of varieties of completely
regular semigroups. To recall the basic contributions to that theory
one should mention the book by Petrich and
Reilly~\cite{Petrich&Reilly:1999}
and a series of seminal papers by
Pol\'ak~\cite{Polak:1985,Polak:1987,Polak:1988}.

With the applications of semigroup theory in the algebraic theory of
regular languages via Eilenberg's correspondence, the focus was
shifted in part to the theory of pseudovarieties of finite semigroups,
that is, classes of finite semigroups that are closed under taking
homomorphic images, subsemigroups, and finite direct products. Some of
the results, for example the mentioned description of the lattice of
all varieties of bands, were translated automatically to the case of
pseudovarieties due to the local finiteness property. However, many
results required the development of new techniques. For example,
results concerning pseudovarieties of finite completely regular
semigroups were obtained by Trotter and the first author
in~\cite{Almeida&Trotter:1999a} using profinite techniques; such
techniques evolved in the aftermath of Reiterman's
\cite{Reiterman:1982} use of profinite semigroups to describe
pseudovarieties and were extensively developed by the first author. For
an overview of known results in this area we refer to the books on the
theory of finite semigroups by the first author~\cite{Almeida:1994a}
and Rhodes and Steinberg~\cite{Rhodes&Steinberg:2009qt}.

Another step in the applications in the theory of regular languages
was a refinement of Eilenberg's correspondence made by
Pin~\cite{Pin:1995a} which captures what classes of languages are
obtained when complementation is discarded. The theory of finite
ordered semigroups introduced there may be viewed as a generalization
of the theory of finite semigroups. Simply put, ordered semigroups are
semigroups enriched by a partial order which is compatible with the
multiplication, also called a stable partial order. So, in particular,
every semigroup together with the equality relation is an ordered
semigroup. Furthermore, every pseudovariety of semigroups $\pv V$
determines a pseudovariety of ordered semigroups $\pv V^o$ by taking
all possible stable partial orders on every semigroup from the
original pseudovariety $\pv V$. The pseudovarieties of the form $\pv
V^o$ are denoted by the same symbol as the original pseudovariety $\pv
V$; we call them {\em selfdual}, since they are characterized by the
property of being closed under taking dually ordered semigroups.
Denote the lattice of all subpseudovarieties of ordered semigroups of
the pseudovariety $\pv U$ by $\mathcal L_o(\pv U)$, and the lattice of
all subpseudovarieties of semigroups of the pseudovariety $\pv V$ by
$\mathcal L(\pv V)$. Then, by the previous observations, $\mathcal
L(\pv V)$ may be viewed as a sublattice of $\mathcal L_o(\pv V)$.
Notice also that the ordered version of the theory of profinite
semigroups was introduced by Pin and Weil in~\cite{Pin&Weil:1996b} and
further developed in the literature.

One can say that, besides the results in connection with the theory of
varieties of regular languages, a systematic study of pseudovarieties
of ordered semigroups was not so far significantly developed. We can
mention only the result by Emery~\cite{Emery:1999}, who described the
lattice of all pseudovarieties of ordered normal bands $\mathcal
L_o(\pv{NB})$, that is bands satisfying the identity $xyzx=xzyx$.
Furthermore, Ku\v ril~\cite{Kuril:2015} studied the lattice of all
pseudovarieties of ordered bands. Although his paper deals with
varieties of ordered bands, it yields the same result for
pseudovarieties, because bands are locally finite. The description of
that lattice was possible since he proved that there are just a few
pseudovarieties of ordered bands which are not selfdual. More
precisely, all of them are subpseudovarieties of the pseudovariety of
normal bands. Thus, by the description of the lattice of
pseudovarieties of (unordered) bands and by the result of Emery,
\cite{Kuril:2015} yields a complete description of the lattice
$\mathcal L_o(\pv{B})$.

Our research is inspired by results from~\cite{Kuril:2015}, as we have
tried to extend the results to a richer class. Another pseudovariety
without new subpseudovarieties of ordered semigroups is the
pseudovariety of all groups $\pv G$, because each finite ordered group
is only trivially ordered. More formally, we have $\mathcal L_o(\pv
G)=\mathcal L(\pv G)$.

The main result of this paper is the following statement.

\begin{MyThm}
  \label{t:main-above-sl}
  Let $\pv V$ be a pseudovariety
  of ordered completely regular semigroups containing all
  semilattices. Then $\pv V$ is selfdual.
\end{MyThm}

Hence, one needs to deal with pseudovarieties of ordered completely
regular semigroups which do not contain all semilattices. It is not
difficult to show that such pseudovarieties contain only normal
orthogroups, which are completely regular semigroups in which
idempotents form normal bands. Thus, we can state that all
pseudovarieties of ordered completely regular semigroups which are not
selfdual are contained in $\mathcal L_o(\pv{NOCR})$, where $\pv{NOCR}$
is the pseudovariety of all finite normal orthogroups. We also prove
the following result.

\begin{MyThm}
  \label{t:main-nocr}
  The lattice $\mathcal L_o(\pv{NOCR})$ is isomorphic to the direct product 
  of the lattices $\mathcal L_o(\pv{NB})$ and $\mathcal L (\pv G)$.
\end{MyThm}

Of course, we can not say that we completely described the lattice
$\mathcal L_o(\pv{CR})$, because there are parts which are not known,
namely $\mathcal L(\pv{CR})$ and particularly $\mathcal L(\pv{G})$.
The lattice $\mathcal L(\pv{CR})$ has been extensively studied and
much of the deeper results about it are based on Pol\'ak's work
\cite{Polak:1985,Polak:1987,Polak:1988}. When the idempotents form a
subsemigroup, we say that the semigroup is \emph{orthodox}. Orthodox
completely regular semigroups are also known as \emph{orthogroups}.
They form a pseudovariety, denoted \pv{OCR}. Pol\'ak's methods yield a
complete description of the lattice $\mathcal{L}(\pv{OCR})$ in terms
of $\mathcal L(\pv{G})$ \cite{Pastijn:1991}. Thus, up to a knowledge
of the lattice of pseudovarieties of groups, the lattice of
pseudovarieties of ordered semigroups $\mathcal L_o(\pv{OCR})$ is
described. In any case, the current knowledge about $\Cl L{\pv{CR}}$
is sufficient to obtain the following applications.

\begin{MyThm}
  \label{t:main-modular}
  The lattice $\Cl L_o(\pv{CR})$ is modular.
\end{MyThm}

\begin{MyThm}
  \label{t:main-capB}
  The correspondence $\pv V\mapsto\pv V\cap\pv B$ defines a complete
  endomorphism of the lattice $\Cl L_o(\pv S)$.
\end{MyThm}

The structure of the paper is the following.
Section~\ref{s:preliminaries} fixes notation and recalls some notions
and results from the literature. Then Theorem~\ref{t:main-above-sl} is
proved in Section~\ref{s:above-sl} while Theorem~\ref{t:main-nocr} is
proved in Section~\ref{s:nocr}. Applications, including
Theorems~\ref{t:main-modular} and~\ref{t:main-capB} are given in
Section~\ref{sec:applications}.

\section{Preliminaries and notation}\label{s:preliminaries}

We assume that the reader is familiar with the basic concepts from the
theory of semigroups. In particular, we need some results from
profinite theory of finite semigroups
(see~\cite{Almeida:1994a,Almeida:2003cshort}) for which the ordered
version was introduced in~\cite{Pin&Weil:1996b}. Besides these
theories we also recall other notions.

For a \emph{quasiorder} $\sigma$ on the set $M$, that is a reflexive
and transitive binary relation, we write also $a\mathrel{\sigma} b$ to
express $(a,b)\in\sigma$. Both infix and suffix notation are also used
for a \emph{partial order}, which is an antisymmetric quasiorder. For
a quasiorder $\sigma$, we denote by $\sigma^d$ the dual quasiorder,
that is $\sigma^d=\{(a,b) \in M\times M \mid (b,a)\in \sigma\}$. Then
$\sigma\cap \sigma^d$ is an equivalence relation on $M$, which is
denoted by $\tilde{\sigma}$; the corresponding quotient set is denoted
by $M/{\tilde{\sigma}}$ and its elements by $[a]$ for each $a\in M$.
Moreover, the quasiorder $\sigma$ naturally induces a partial order on
the quotient set $M/{\tilde{\sigma}}$, denoted by~$\le_\sigma$; it is
such that $[a] \le_\sigma [b]$ if and only if $(a,b)\in \sigma$.

A quasiorder $\le$ on a semigroup $S$ is \emph{stable} if, for every
$a,s,t\in S$, the following condition holds: if $s\le t$ then $sa\le
ta$ and $as\le at$. By an \emph{ordered semigroup} we mean a semigroup
$(S,\cdot)$ which is equipped with a stable partial order.
\emph{Homomorphisms of ordered semigroups} are homomorphisms of
semigroups which are isotone. Similar to the unordered case,
\emph{pseudovarieties of ordered semigroups} are classes of finite
ordered semigroups closed under taking homomorphic images, (ordered)
subsemigroups and finite products. For a pseudovariety $\pv V$ of
finite ordered semigroups, we denote by $\pv V^d$ the pseudovariety of
all dually ordered semigroups, i.e., $(S,\cdot,\le)\in\pv V^d$ if and
only if $(S,\cdot,\ge)\in\pv V$. We say that $\pv V$ is {\em selfdual}
if $\pv V=\pv V^d$. Such a pseudovariety of ordered semigroups is also
characterized by the property that $(S,\cdot,\le) \in\pv V$ implies
$(S,\cdot,=) \in\pv V$: the identity mapping is a homomorphism from
the ordered semigroup $(S,\cdot,=)$ onto an arbitrary ordered
semigroup $(S,\cdot,\ge)$, while the diagonal mapping embeds
$(S,\cdot,{=})$ into the product
$(S,\cdot,{\le})\times(S,\cdot,{\ge})$. If a pseudovariety of
semigroups $\pv V$ is given, then we may consider the (selfdual)
pseudovariety consisting of all ordered semigroups which are members
of $\pv V$ equipped with every possible stable partial order.
Conversely, if $\pv V$ is a selfdual pseudovariety of ordered
semigroups then we may consider the pseudovariety consisting of all
semigroups $(S,\cdot)$ such that $(S,\cdot,=) \in\pv V$. These
constructions give a one-to-one correspondence between selfdual
pseudovarieties of ordered semigroups and pseudovarieties of
semigroups. Note that usually the same symbol is used to denote both a
pseudovariety of semigroups and the corresponding selfdual
pseudovariety of ordered semigroups.

By a \emph{profinite semigroup} we mean a compact semigroup which is
residually finite as a topological semigroup, finite semigroups being
viewed as discrete topological spaces. Just as pseudovarieties (of
semigroups) are defined by sets of \emph{pseudoidentities}
\cite{Reiterman:1982}, pseudovarieties of ordered semigroups are
defined by \emph{pseudoinequalities}
\cite{Molchanov:1994,Pin&Weil:1996b}. The difference being semantic,
syntactically both pseudoidentities and pseudoinequalities are pairs
of \emph{pseudowords}. In a sense pseudowords are terms in an enriched
language where not only the semigroup operation is allowed; more
formally, pseudowords are members of a finitely generated free
profinite semigroup. Often, we only need an extra unary operation,
represented by the pseudoword $x^\omega$; in general, for an element
$s$ of a profinite semigroup, $s^\omega$ is the only idempotent in the
closed subsemigroup generated by~$s$. We also write $s^{\omega+1}$
instead of $s^\omega s$. It lies in the maximal subgroup with
idempotent $s^\omega$; its inverse in that group is denoted
$s^{\omega-1}$. For a set $\Sigma$ of pseudoidentities or
pseudoinequalities, $\op\Sigma\cl$ is the class consisting of all
finite semigroups that satisfy all elements of~$\Sigma$ in the sense
that, under arbitrary evaluation of the variables, the two sides are,
respectively, equal or in increasing order.

Completely regular semigroups may be viewed as semigroups with a unary
operation of inversion in the unique maximal subgroup containing a
given element. As such, they form a variety, which is defined by the
identities
$$x(yz)=(xy)z,
\quad xx^{-1}x=x,
\quad (x^{-1})^{-1}=x,
\quad xx^{-1}=x^{-1}x.
$$
It is common to write $s^0$ for the product $ss^{-1}$. In a completely
regular profinite semigroup, note that $s^{-1}=s^{\omega-1}$ and
$s^0=s^\omega$.

We extend the notion of ordered semigroup to unary semigroups by
requiring that not only the partial order be stable under
multiplication but also under the unary operation. Note that this is
not the common practice in the theory of ordered groups, where
inversion is not required to preserve the order. But, it is natural to
assume it as a generalization of the finite case, where inversion is
given by a power. Also, when studying varieties of ordered algebras,
it is natural to require that all basic operations preserve the order
\cite{Bloom:1976}. In particular, note that, for a stable quasiorder
$\rho$ on a completely regular semigroup the conditions
$s\mathrel{\rho}s^0$ and $s^0\mathrel{\rho}s$ are equivalent: taking
inverses in $s\mathrel{\rho}s^0$, we get $s^{-1}\mathrel{\rho}s^0$,
which, multiplying by $s$, yields $s^0\mathrel{\rho}s$; the same steps
also establish the converse.

All finite completely regular semigroups form the pseudovariety
$\pv{CR}=\op x^{\omega+1}=x\cl$ and, therefore, all finite orthogroups
form the pseudovariety $\pv{OCR}=\op x^{\omega+1}=x, x^\omega y^\omega
= (x^\omega y^\omega)^\omega\cl$. This pseudovariety contains the
pseudovariety of all finite bands $\pv B=\op x^2=x\cl$ and, in
particular, all finite semilattices $\pv {Sl}=\op x^2=x, xy=yx\cl$ and
all finite normal bands $\pv{NB}=\op x^2=x, xyzx=xzyx \cl$. The
pseudovariety $\pv{OCR}$ also contains the pseudovariety of all finite
groups $\pv G=\op x^\omega=1\cl$. We also follow the standard notation
where $\pv{\overline H}$ denotes the pseudovariety of all semigroups
whose subgroups belong to a given pseudovariety of groups $\pv H$. We
let $\pv{CR(H)}=\pv{CR}\cap \overline{\pv H}$.

Our aim is a description of all subpseudovarieties of ordered
semigroups of the pseudovariety $\pv{CR}$. By Birkhoff's basic results
on universal algebra, varieties of algebras can be uniquely determined
by fully invariant congruences on the term algebra over a countable
set of variables. A similar concept of fully invariant stable
quasiorder is possible in the case of varieties of ordered algebras.
In particular, a proof technique in~\cite{Kuril:2015} is based on
manipulating the fully invariant stable quasiorders corresponding to
the varieties of ordered bands. However, the situation in finite
universal algebra is more complicated. One needs to study relatively
free profinite semigroups over a finite set of variables and,
moreover, one needs to consider them for arbitrarily large finite sets.
We use the notation $\Om X S$ from~\cite{Almeida:1994a} for the free
profinite semigroup over a given finite set of generators $X$ in the
pseudovariety of all semigroups $\pv S$. Usually, we use the natural
number $n$ as an index in $\Om X S$ instead of $X$ when
$X=\{x_1,\dots,x_n\}$. Then the inclusion $\Om n S\subseteq \Om m S$
may be used freely. The same structures are also considered as free
profinite ordered semigroups. Therefore, we associate with a
pseudovariety $\pv V$ of finite ordered semigroups a system of
relations $(\rho_{\pv V, n})_{n}$ on the profinite
semigroups $\Om n S$ where each $\rho_{\pv V, n}$ is given by the
formula
$$\rho_{\pv V, n} %
=\{(u,v) \in \Om n S \times \Om n S \mid \pv V\models u\le v\};$$ %
note that we write $\pv V\models u\le v$ to mean that every member of~\pv
V satisfies the pseudoinequality $u\le v$. It is known that each
$\rho_{\pv V , n} $ in this system is a closed stable quasiorder. It
is also fully invariant, but this property is satisfied by the whole
system in the following sense, which we introduce for a general system
of relations.

By a {\em system of relations} $(\rho_n)_{n}$ we mean a family of
relations indexed by the positive integers such that $\rho_n$ is a
relation on $\Om n S$ for every~$n$. We say that the system
$(\rho_n)_{n}$ is \emph{fully invariant} if, for each continuous
homomorphism $\varphi : \Om n S \rightarrow \Om m S$ and $u
\mathrel{\rho_n} v$, we have $\varphi(u) \mathrel{\rho_m} \varphi(v)$.
Additionally, if every relation $\rho_n$ in this system is a closed
stable quasiorder on $\Om n S$, then we call the system
$\rho=(\rho_n)_{n}$ a {\em fully invariant system of
  closed stable quasiorders}. It is not clear whether every such
system $\rho$ determines a pseudovariety $\pv V$ such that
$\rho=(\rho_{\pv V, n})$ (see~\cite{Almeida&Klima:2017a} for a
discussion concerning a related conjecture). To explain what kind of
property is potentially missing, first denote
$\tilde{\rho}=(\widetilde{\rho_n})_{n}$, where
$\widetilde{\rho_n}$ is the equivalence relation corresponding to
$\rho_n$. Then, for each $n$ we consider $\Om n
S/{\widetilde{\rho_n}}$, which is a compact ordered semigroup, where
the partial order is $\le_{\rho_n}$. If for each $n\ge1$ the
ordered semigroup $\Om n S/{\widetilde{\rho_n}}$ is residually finite,
then $\rho=\rho_{\pv V}$, where $\pv V$ is given as the class of all
finite ordered semigroups which are finite quotients of some ordered
semigroup $(\Om n S/{\widetilde{\rho_n}},\cdot,\le_{\rho_n})$. If this
property is true then we call $\rho$ a {\em complete system of
  pseudoinequalities}. Note that, for the corresponding pseudovariety
$\pv V$ of ordered semigroups, we may also write $\pv V = \op \rho \cl$
to mean $\pv V=\op\rho_n: n\ge1\cl$. We talk about a {\em complete
  system of pseudoidentities} when $\pv V$ is a selfdual pseudovariety
of ordered semigroups. Notice also that $\pv V$ is selfdual if and
only if each $\rho_n$ is symmetric.

The question whether the ordered semigroup $\Om n
S/{\widetilde{\rho_n}}$ is residually finite seems to be difficult in
general, but it is trivial in the case when $\Om n
S/{\widetilde{\rho_n}}$ is itself finite. This simple observation is
sufficient for our considerations, because it turns out to be enough
to work with systems $\rho=(\rho_n)$ for which each equivalence
relation $\widetilde{\rho_n}$ has finite index. Note that this
condition is equivalent to the pseudovariety $\op\rho\cl$ being
locally finite.

For two fully invariant systems of closed stable quasiorders $\rho$
and $\sigma$, we write $\rho\subseteq \sigma$ if $\rho_n\subseteq
\sigma_n$ holds for every $n\ge1$. For specific pseudovarieties \pv V,
such as $\pv{CR}$, $\pv{H}$, $\pv{CR(H)}$, $\pv B$, $\pv{Sl}$ etc., we
will use the symbol $\sim_{\pv V}$ instead of $\rho_{\pv V}$. In general, if
we want to express that a pair $u,v\in\Om n S$ is $\rho_n$-related
then we may omit the index $n$ which is clear from the context and we
may simply write $u\mathrel{\rho} v$.
Note that $\rho_{\pv V}\subseteq \rho_{\pv W}$ whenever $\pv
W\subseteq \pv V$. In this way, we may write, e.g.,
${\sim_{\pv{CR}}}\subseteq{\sim_{\pv{OCR}}}\subseteq
{\sim_{\pv{B}}}\subseteq{\sim_{\pv{Sl}}}$.

We recall the definition of
the functions $0$ and $1$, which are fundamental in the theory of
completely regular semigroups. Notice that these functions may be
defined in a more general setting; however, in our contribution, we
use them just for the free profinite semigroups $\Om X S$ over the
finite set $X$.
For an arbitrary $u\in \Om XS$, we denote by $\cont(u)$ the smallest
set $A\subseteq X$ such that $u$ belongs to the closed subsemigroup of
$\Om X S$ generated by $A$. In other words, it is the set of all
letters occurring in $u$, i.e., the set of all $x\in X$ such that
there exist $u', u''\in\Om XS^1$ satisfying $u=u'xu''$; for this
reason, it is called the \emph{content} of~$u$. We extend the previous
definition of the content function $\cont$ to the case of the free
profinite monoid $\Om{X}{S}^1$ by putting $\cont (1)=\emptyset$.
Recall that $u \sim_{\pv{Sl}} v$ if and only if $\cont(u)=\cont(v)$.
Now, for every $u\in\Om X S$, there are unique $u',u''\in\Om{X}{S}^1$
and $x\in X$ such that $u=u'xu''$, $x\not\in \cont(u')$ and
$\cont(u')\cup\{x\}=\cont(u)$. The pseudoword $u'$ in this unique
factorization is denoted by $0(u)$ and $x$ in the same factorization
is denoted by $\overline 0(u)$. Note that for $u$ such that
$|\cont(u)|=1$ we have $0(u)=1$. By left to right duality we obtain
the definition of $1(u)$ and $\overline 1(u)$.
See~\cite{Almeida&Trotter:1999a} for further details.

A fundamental tool in semigroup theory is given by Green's relations
\cite{Howie:1995}. For two elements $s$ and $t$ of a semigroup, we
write $s\le_{\Cl J}t$, $s\le_{\Cl L}t$, $s\le_{\Cl R}t$ if,
respectively, the ideal, the left ideal, the right ideal generated by
$s$ is contained in that generated by~$t$. These are quasiorders on
the semigroup, whose intersection with their duals are denoted,
respectively, \Cl J, \Cl L, and \Cl R. The intersection of \Cl R and
\Cl L is denoted \Cl H. A semigroup is said to be \emph{stable} if
$s\mathrel{\Cl J}st$ implies $s\mathrel{\Cl R}st$ and $s\mathrel{\Cl
  J}ts$ implies $s\mathrel{\Cl L}ts$. It is well known that, in a
stable semigroup, the equivalence relation \Cl J is the join of \Cl R
and \Cl L; moreover, the main semigroups of interest in this paper,
namely profinite semigroups and completely regular semigroups, are
stable \cite{Koch&Wallace:1957}.

Given a congruence $\theta$ on a semigroup $S$ and one of Green's
relations \Cl K, we say that the elements $s$ and $t$ of~$S$ are \Cl
K-equivalent modulo $\theta$ if their $\theta$-classes are \Cl
K-equivalent in the quotient semigroup $S/\theta$. It is well known
that, if \pv V is a pseudovariety between \pv{Sl} and \pv{CR}, two
pseudowords $u,v\in\Om XS$ are \Cl J-equivalent modulo~\pv V if and
only if $\cont(u)=\cont(v)$ (see \cite[Theorem~8.1.7]{Almeida:1994a});
it follows that, in this case, we have $u\mathrel{\Cl R}v$ modulo~\pv
V if $0(u)=0(v)$, and $u\mathrel{\Cl L}v$ modulo~\pv V if $1(u)=1(v)$.

The following is an easy consequence
of~\cite[Proposition~2.2]{Pastijn&Trotter:1988} which gives the
analogous result for varieties.
  
\begin{Prop}
  \label{p:locally-finite-H}
  Let $\pv H$ be a locally finite pseudovariety of groups. Then
  $\pv{CR(H)}$ is a locally finite pseudovariety of semigroups.
\end{Prop}

It is well known that every pseudovariety is the directed union of its
subpseudovarieties generated by a single semigroup, which are locally
finite. For this purpose, it suffices to consider a countable sequence
of all (up to isomorphism) finite members of the pseudovariety and
then for each natural number $n$ to take the pseudovariety generated
by the direct product of the first $n$ semigroups in the sequence.
Such an expression of $\pv G$ is used later.

\section{Pseudovarieties above semilattices}
\label{s:above-sl}

In this section, we prove that above the pseudovariety $\pv{Sl}$ all
pseudovarieties of ordered completely regular semigroups are selfdual.

So, let $\pv V$ be a given pseudovariety of ordered completely regular
semigroups with the corresponding fully invariant system of closed
stable quasiorders $\rho_{\pv V}$. Our goal is to show that all
relations $(\rho_{\pv V} )_n$ are symmetric, which is equivalent with
the selfduality of $\pv V$.

Let $\pv H=\pv V\cap \pv G$. Thus, we assume that ${\sim_{\pv{CR(H)}}}
\subseteq \rho_{\pv V} \subseteq {\sim_{\pv{Sl}}}$. Although some of
the following results are satisfied for an arbitrary pseudovariety
$\pv V$, we further assume that $\pv H=\pv V\cap \pv G$ is locally
finite. Hence, $\pv V \subseteq \pv{CR(H)}$ is also locally finite by
Proposition~\ref{p:locally-finite-H}.

Following the basic idea of the paper~\cite{Kuril:2015}, for a
complete system of pseudo\-inequalities $\rho$, we consider the system
of relations $\rho^0=(\rho^0_n)_{n}$ where each $\rho^0_n$ is
defined as follows. On the set $\Om n S$ we let $u
\mathrel{\rho^0_n} v$ if there are pseudowords $u', v'\in \Om{n+1}S$
such that $u'\mathrel{\rho_{n+1}} v'$, $0(u')=u$ and $0(v')=v$.
Furthermore, we let $\rho^{0\cont}=(\rho^{0\cont}_n)_{n}$ where each
$\rho^{0\cont}_n$ is the intersection $\rho^0_n\cap
(\sim_{\pv{Sl}})_n$. Dually, we define $\rho^1$ and $\rho^{1\cont}$
if we consider the function $1$ instead of $0$.

\begin{Lemma}
  \label{l:rho0}
  Let $\rho$ be a complete system of pseudoinequalities such that
  $\rho\subseteq{\sim_{\pv{Sl}}}$ and let $u,v\in\Om XS$ be
  pseudowords such that $\cont(u)=\cont(v)$. If\/
  $0(u)\mathrel{\rho^0}0(v)$, then there exist $x,x'\in X$ and
  $s,s'\in\Om XS^1$ such that, for $u'=0(u)xs$ and $v'=0(v)x's'$, the
  following properties hold:
  \begin{align}
    & u' \mathrel{\rho} v',
      \quad 0(u')=0(u),
      \quad 0(v')=0(v)\,
      \label{e:zero-and-rho} \\
    & \cont(u)=\cont(u')=\cont(v'), 
      \quad \overline 0(u)=x,
      \quad \overline 0(v)=x' \,.
      \label{e:content-vs-zero}
  \end{align}
\end{Lemma}

\begin{proof}
  From the assumption that $0(u)\mathrel{\rho^0}0(v)$ we know that
  there exist $x,x'\in X$ and $s,s'\in\Om XS^1$ such that, for
  $u'=0(u)xs$ and $v'=0(u)x's'$, the conditions
  in~\eqref{e:zero-and-rho} hold. We distinguish two cases. At first,
  we assume that $\cont(0(u))=\cont(0(v))$. Then $x$ is equal to $x'$
  because $\rho\subseteq {\sim}_{\pv{Sl}}$. Since $\rho$ is fully
  invariant and $\cont(u)=\cont(v)$, we may assume that $x$ belongs to
  $\cont(u)\setminus\cont(0(u))$. Consequently, the conditions
  in~\eqref{e:content-vs-zero} hold. Secondly, we assume that
  $\cont(0(u))\not=\cont(0(v))$. Then $x$ occurs in $v'$ and,
  therefore, in $0(v)$. In a similar way we get that $x'$ occurs in
  $0(u)$. Since $\cont(u')=\cont(v')$, we deduce
  \eqref{e:content-vs-zero} again.\qed
\end{proof}

The following result relates $\rho^0$ and $\rho^1$ with the
restriction of $\rho$ to the set of idempotents, the so-called trace
of~$\rho$.

\begin{Lemma}
  \label{l:trace}
  Let $\rho$ be a complete system of pseudoinequalities such that
  ${\sim_{\pv{CR}}}\subseteq\rho\subseteq{\sim_{\pv{Sl}}}$ and let
  $u,v\in\Om XS$ be arbitrary pseudowords. Then
  $u^\omega\mathrel{\rho}v^\omega$ if and only if
  \begin{equation}
    \label{eq:trace}
    0(u)\mathrel{\rho^0} 0(v), \quad
    1(u) \mathrel{\rho^1} 1(v),\quad
    \text{and} \quad u\sim_{\pv{Sl}}v\,.
  \end{equation}
\end{Lemma}

\begin{proof}
  Suppose first that $u^\omega\mathrel{\rho}v^\omega$. Since, in
  general, $w$ and $w^\omega=ww^{\omega-1}=w^{\omega-1}w$ have the
  same values for each of the functions $c$, $0$, and~$1$, and
  $\rho\subseteq{\sim_{\pv{Sl}}}$, we obtain the conditions
  in~\eqref{eq:trace}.

  Conversely, suppose that~\eqref{eq:trace} holds. Let $u',v'\in\Om
  XS$ be given by Lemma~\ref{l:rho0} and let $u'',v''\in\Om XS$ be
  given by the dual version of Lemma~\ref{l:rho0}, concerning
  $\rho^1$, so that there are $y,y'\in X$ and $t,t'\in\Om XS^1$ such
  that $u''=ty1(u)$, $v''=t'y'1(v)$, and
  \begin{align}
    & u'' \mathrel{\rho} v'',
      \quad 1(u'')=1(u),
      \quad 1(v'')=1(v)
      \label{e:one-and-rho} \\
    & \cont(u)=\cont(u'')=\cont(v''),
      \quad \overline 1(u)=y,
      \quad \overline 1(v)=y'\,. 
    \label{e:content-vs-one}
  \end{align}
  From the first conditions in~\eqref{e:zero-and-rho}
  and~\eqref{e:one-and-rho}, we deduce that
  \begin{equation}
    \label{eq:trace-idempotents}
    (u'u'')^\omega\mathrel{\rho}(v'v'')^\omega.
  \end{equation}
  But, by~\eqref{e:content-vs-zero} and~\eqref{e:content-vs-one}, the
  idempotents $(u'u'')^\omega$ and $u^\omega$ have the same content
  and, therefore, they are $\mathcal{J}$-equivalent modulo
  $\sim_{\pv{CR}}$. By~\eqref{e:zero-and-rho}
  and~\eqref{e:one-and-rho}, they also have the same values under each
  of the functions $0$ and~$1$, which entails that they are
  $\mathcal{H}$-equivalent modulo $\sim_{\pv{CR}}$, whence they are in
  the same $\sim_{\pv{CR}}$-class. For the same reason, we also have
  $(v'v'')^\omega\sim_{\pv{CR}}v^\omega$. Combining these observations
  with~\eqref{eq:trace-idempotents}, we obtain the required relation
  $u^\omega\mathrel{\rho}v^\omega$ since
  ${\sim_{\pv{CR}}}\subseteq\rho$.\qed
\end{proof}

The following lemma enables us to prove that the relations $\rho_n$
are symmetric by a certain inductive argument using the symmetry of
both $\rho^0_n$ and $\rho^1_n$. It is the order analog for finite
semigroups of a key and well known result in the theory of completely
regular semigroups \cite[Lemma~VI.3.1]{Petrich&Reilly:1999}.

\begin{Lemma}
  \label{l:trace-kernel}
  Let $S$ be a completely regular semigroup and let $\rho$ be a stable
  quasiorder on~$S$. Then, the following conditions are equivalent for
  arbitrary \Cl J-equivalent elements $s,t\in S$:
  \begin{enumerate}[(i)]
  \item\label{item:trace-kernel-1} $s\mathrel{\rho}t$;
  \item\label{item:trace-kernel-2} $s^0\mathrel{\rho}t^0$ and
    $st^{-1}\mathrel{\rho}(st^{-1})^0$;
  \item\label{item:trace-kernel-3} $s^0\mathrel{\rho}t^0$ and
    $t^{-1}s\mathrel{\rho}(t^{-1}s)^0$;
  \item\label{item:trace-kernel-4} $s^0\mathrel{\rho}t^0$ and
    $s^{-1}t\mathrel{\rho}(s^{-1}t)^0$;
  \item\label{item:trace-kernel-5} $s^0\mathrel{\rho}t^0$ and
    $ts^{-1}\mathrel{\rho}(ts^{-1})^0$.
  \end{enumerate}
\end{Lemma}

\begin{proof}
  We establish only the equivalence
  (\ref{item:trace-kernel-1})~$\Leftrightarrow$~(\ref{item:trace-kernel-2})
  as the equivalence of~(\ref{item:trace-kernel-1}) with each of the
  remaining conditions may be proved similarly.
  
  Suppose first that $s\mathrel{\rho}t$. Since $\rho$ is stable, it
  follows that $s^0\mathrel{\rho} t^0$. On the other hand, multiplying
  on the right $s\mathrel{\rho}t$ by $t^{-1}$, we obtain
  $st^{-1}\mathrel{\rho}t^0$. Further multiplying on the left by
  $(st^{-1})^0$, we deduce that $st^{-1}\mathrel{\rho}(st^{-1})^0$.

  For the converse, assuming~(\ref{item:trace-kernel-2}), note that
  $s=s^0ss^0 \mathrel{\rho}t^0 st^0 =t^0 st^{-1}t
  \mathrel{\rho}t^0(st^{-1})^0 t$. Since $t^0$ and $(st^{-1})^0$ are
  \Cl J-equivalent idempotents that admit as a common suffix the
  element $t$ of the same \Cl J-class, by stability of $S$, they are
  $\mathcal{L}$-equivalent modulo $\sim_{\pv{CR}}$. Hence, we have
  $t^0(st^{-1})^0 t=t^0t=t$ and, therefore, $s\mathrel{\rho}t$.\qed
\end{proof}

\begin{Lemma}
  \label{l:characterizing-rho}
  Let $\rho$ be a complete system of pseudoinequalities such that
  ${\sim_{\pv{CR}}} \subseteq\rho \subseteq {\sim_{\pv{Sl}}}$ and let
  $u,v\in\Om X S$ be arbitrary pseudowords. The following conditions
  are equivalent:
  \begin{enumerate}[(i)]
  \item  $u\mathrel{\rho} v$;
  \item $0(u)\mathrel{\rho^0} 0(v)$, $1(u) \mathrel{\rho^1} 1(v)$,
    $u\sim_{\pv{Sl}}v$, and $uv^{\omega-1}\mathrel{\rho}
    (uv^{\omega-1})^\omega$;
  \item $0(u)\mathrel{\rho^0} 0(v)$, $1(u) \mathrel{\rho^1} 1(v)$,
    $u\sim_{\pv{Sl}}v$, and $vu^{\omega-1}\mathrel{\rho}
    (vu^{\omega-1})^\omega$.
  \end{enumerate}
\end{Lemma}

\begin{proof}
  The result is an immediate consequence of
  Lemmas~\ref{l:trace} and~\ref{l:trace-kernel}.\qed
\end{proof}

The next lemma describes basic properties of the relation
$\rho^{0\cont}$ which is equal to $\rho^0$ under the asumption that
$\rho^0\subseteq{\sim}_{\pv{Sl}}$. It enables an induction argument in
the main statement of this section.

\begin{Lemma}
  \label{l:rho-0-c-is-ficsq}
  Let $\pv H$ be a locally finite pseudovariety of groups and $\rho$
  be a complete system of pseudoinequalities such that
  ${\sim}_{\pv{CR(H)}} \subseteq\rho \subseteq {\sim}_{\pv{Sl}}$. Then
  $\rho^{0\cont}$ is a complete system of pseudoinequalities which
  contains $\rho$.
\end{Lemma}

\begin{proof}
  First, we prove that for each $n$ the relation $\rho^{0\cont}_n$
  is stable. Let $u\mathrel{\rho^{0\cont}} v$ be a pair of pseudowords
  and $w$ be another pseudoword and assume that $u,v,w\in \Om{n}S$.
  Consider the set $A=\cont (uw)\subseteq \{x_1,\dots , x_n\}$. There
  exist $k\in\{1,\dots, n+1\}$ and $s,s'\in\Om {n+1}S^1$ such that
  $$ux_ks \mathrel{\rho_{n+1}} vx_ks',%
  \quad x_k\not\in \cont(u)=\cont(v),%
  \quad x_k=\overline{0}(ux_ks)=\overline{0}(vx_ks')\,.$$ %
  Using the assumption that $\rho$ is a fully invariant system, we may
  assume that $k=n+1$. Since $\rho_{n+1}$ is stable, we get
  $wux_{n+1}s \mathrel{\rho_{n+1}} wvx_{n+1}s'$ from which
  $wu\mathrel{\rho^{0\cont}_n} wv$ follows. To prove
  $uw\mathrel{\rho^{0\cont}_n} vw$, we need to substitute $wx_{n+1}$
  for $x_{n+1}$ in $ux_{n+1}s \mathrel{\rho_{n+1}} vx_{n+1}s'$ first,
  which is possible because $\rho$ is fully invariant.

  Next, we prove that, for each $n$, the relation $\rho^{0\cont}_n$
  is a quasiorder. The relation $\rho^{0\cont}_n$ is clearly
  reflexive. To prove transitivity, assume that we have
  $u\mathrel{\rho^{0\cont}_n}v$ and $v\mathrel{\rho^{0\cont}_n}w$.
  Since $\rho^{0\cont} \subseteq\ \sim_{\pv{Sl}}$ we have
  $\cont(u)=\cont(v)=\cont(w)$. Thus, there are $k\in\{1,\dots,n+1\}$
  and $s,s'\in\Om {n+1}S^1$ such that
  $$ux_ks \mathrel{\rho_{n+1}} vx_ks',%
  \quad x_k\not\in \cont(u)=\cont(v),%
  \quad x_k=\overline{0}(ux_ks)=\overline{0}(vx_ks')\,.$$ %
  We have similarly a letter $x_{\ell}$ and pseudowords $t,t'\in\Om
  {n+1}S^1$ for the relation of pseudowords
  $v\mathrel{\rho^{0\cont}_n}w$. Since $\rho$ is fully invariant, we
  may assume that $x_{\ell}=x_k$. Thus, we get
  $$vx_kt \mathrel{\rho_{n+1}} wx_kt',\  x_k\not\in \cont(v)=\cont(w),\ 
  x_k=\overline{0}(vx_kt)=\overline{0}(wx_kt')\, .$$ If we use
  stability of $\rho$, we get $(ux_ks)^\omega \mathrel{\rho_{n+1}}
  (vx_ks')^\omega$ and $(vx_kt)^\omega \mathrel{\rho_{n+1}}
  (wx_kt')^\omega$. If we multiply the first relation by
  $(vx_kt)^\omega $ on the right, we obtain $(ux_ks)^\omega
  (vx_kt)^\omega \mathrel{\rho_{n+1}} (vx_ks')^\omega(vx_kt)^\omega$.
  Since $0((vx_ks')^\omega)=0((vx_kt)^\omega)=v$, the pseudowords
  $(vx_ks')^\omega$ and $(vx_kt)^\omega$ are \Cl R-equivalent
  idempotents modulo~${\sim}_{\pv{CR(H)}}$. It follows that their
  product $(vx_ks')^\omega(vx_kt)^\omega$ is
  ${\sim}_{\pv{CR(H)}}$-equivalent to the right factor
  $(vx_kt)^\omega$. Thus, we have $(ux_ks)^\omega (vx_kt)^\omega
  \mathrel{\rho_{n+1}} (vx_kt)^\omega$, because
  ${\sim}_{\pv{CR(H)}}\subseteq\rho$. By transitivity of $\rho$ and
  from $(vx_kt)^\omega \mathrel{\rho_{n+1}} (wx_kt')^\omega$, we get
  $(ux_ks)^\omega (vx_kt)^\omega \mathrel{\rho_{n+1}}
  (wx_kt')^\omega$. By applying the $0$ function, we may deduce that
  $u\mathrel{\rho^{0\cont}_n}w$.

  The next step is to show that, for each $n$, the relation
  $\rho^{0\cont}_n$ is closed. Let $(u_k,v_k)_k$ be a sequence of
  pairs of pseudowords in~$\rho^{0\cont}_n$ such that $\lim u_k=u$
  and $\lim v_k=v$. Since $u_k \mathrel{(\sim_{\pv{Sl}})_n} v_k$ holds
  for every~$k$, we deduce that $u \mathrel{(\sim_{\pv{Sl}})_n} v$. By
  the assumption that $u_k\mathrel{\rho^0_n} v_k$, there are
  $\ell_k\in\{1,\dots,n+1\}$ and pseudowords $s_k,s_k'\in\Om {n+1}
  S^1$ such that
  $$u_kx_{\ell_k}s_k \mathrel{\rho_{n+1}} v_kx_{\ell_k}s'_k,\ %
  x_{\ell_k}\not\in \cont(u_k)=\cont(v_k),\ %
  x_{\ell_k} %
  =\overline{0}(u_kx_{\ell_k}s_k) %
  = \overline{0}(v_kx_{\ell_k}s'_k)\,.$$ %
  We consider an infinite set $I$ of indices such that
  $\cont(u_k)=\cont(u)$ for every $k\in I$. Then, we may assume that,
  for all $k\in I$, the variable $x_{\ell_k}$ is equal to $x_{n+1}$
  because $\rho_{n+1}$ is fully invariant. Now, from the sequence of
  pairs $(u_kx_{\ell_k}s_k,v_kx_{\ell_k}s'_k)_{k\in I}$ we can choose
  a convergent subsequence; denote by $(\bar u,\bar v)$ its limit.
  Since $\rho_{n+1}$ is a closed relation, we obtain $ \bar
  u\mathrel{\rho_{n+1}} \bar v$. It follows that $u=0(\bar{u})
  \mathrel{\rho^0_n} 0(\bar{v})=v$ because multiplication is a
  continuous operation on $\Om {n+1}S$.

  We also need to prove that $\rho^{0\cont}$ is fully invariant.
  Assume that we have $u\mathrel{\rho^{0\cont}_n}v$ and consider a
  continuous homomorphism $\varphi: \Om nS \rightarrow \Om mS$. Again,
  $\rho^{0\cont} \subseteq\ \sim_{\pv{Sl}}$ gives $\cont(u)=\cont(v)$
  and, moreover, there are $s,s'\in\Om {n+1}S^1$ such that
  $$ux_{n+1}s \mathrel{\rho_{n+1}} vx_{n+1}s',\ %
  x_{n+1}\not\in \cont(u)=\cont(v),\ %
  x_{n+1}%
  =\overline{0}(ux_{n+1}s)%
  =\overline{0}(vx_{n+1}s')\,.$$ %
  Consider the continuous homomorphism $\varphi' : \Om {n+1}S
  \rightarrow \Om {m+1}S$ defined by the rules
  $\varphi(x_{n+1})=x_{m+1}$ and $\varphi'(x_i)=\varphi(x_i)$ for all
  indices $i\in\{1,\dots , n\}$. It follows that $\varphi'(ux_{n+1}s)
  \mathrel{\rho_{m+1}} \varphi'(vx_{n+1}s')$. Since
  $0(\varphi'(ux_{n+1}s))=\varphi'(u)=\varphi(u)$ and, similarly, we
  have $0(\varphi'(vx_{n+1}s'))=\varphi(v)$, we obtain $\varphi(u)
  \mathrel{\rho_{m}} \varphi(v)$. We have proved that $\rho^{0\cont}$
  is fully invariant.

  Finally, assuming that $u \mathrel{\rho_n} v$, we see that $u$ and
  $v$ have the same content and $ux_{n+1}\mathrel{\rho_{n+1}}
  vx_{n+1}$. Hence, $u \mathrel{\rho^{0\cont}_n} v$ and we get the
  inclusion $\rho_n \subseteq \rho^{0\cont}_n$. So, we have proved
  that $\rho^{0\cont}$ is a fully invariant system of closed stable
  quasiorders such that $\rho\subseteq \rho^{0\cont}$. Since, for each
  $n$, the equivalence relation $\widetilde{\rho_n}$ has finite index,
  the same is true for $\widetilde{\rho^{0\cont}_n}$. Consequently,
  the quotient $\Om n S/{\widetilde{\rho^{0\cont}_n}}$ is finite and
  $\rho^{0\cont}$ is a complete system of pseudoinequalities.\qed
\end{proof}

The previous lemma can be applied in the case where $\rho^0 \subseteq
{\sim}_{\pv{Sl}}$. Since the constructed system of relations
$\rho^0$ may not be contained in $\sim_{\pv{Sl}}$, we need to
understand what happens in such a case.

Note that the inclusion $\rho \subseteq \rho^0$, which was proved in
the previous lemma,
is true in general and does not require the assumptions of the lemma.

\begin{Lemma}
  \label{l:case-rho-0-anti-content}
  Let $\rho$ be a complete system of pseudoinequalities such that
  ${\sim}_{\pv{CR}} \subseteq\rho \subseteq {\sim}_{\pv{Sl}}$. If
  $\rho^0$ is not contained in $\sim_{\pv{Sl}}$, then all relations
  $\rho^0_n$ are symmetric.
\end{Lemma}

\begin{proof}
  Let $u \mathrel{\rho^0_n} v$ hold for a given arbitrary pair of
  pseudowords from $\Om nS$. We want to show that also $v
  \mathrel{\rho^0_n} u$. Assume for the moment, that
  $\cont(u)\not=\cont(v)$. Since we obtain the pair $(u,v)$ by
  application of the function $0$ on a certain pair of
  $\rho_{n+1}$-related pseudowords which have the same content, there
  are indices $k,\ell\in\{1,\dots ,n\}$ such that
  \begin{equation}\label{e:rho-zero-outside-content}
    \cont(ux_k)=\cont(vx_{\ell}),\quad 0(ux_k)=u\quad 
    \text{and}\quad 0(vx_\ell)=v\,.
  \end{equation} 
  Note that, whenever $\cont(u)=\cont(v)$, one can take for $x_k$ and
  $x_\ell$ the letter $x_{n+1}$ and (\ref{e:rho-zero-outside-content})
  also holds. The pair of pseudowords $ux_k$, $vx_\ell$ is useful in
  the following considerations.

  We distinguish two cases. First, assume that there is a pair of
  pseudowords $s$ and $t$ such that $s\mathrel{\rho_m} t$, the first
  letter of $s$ is $y$, the first letter of $t$ is $z$, and that these
  letters are different elements from $\{x_1,\dots, x_{m}\}$. Now, we
  substitute in $s$ and $t$ the pseudoword $vx_\ell$ for $y$, the
  pseudoword $ux_k$ for $z$, and $x_{k}$ for other variables. In this
  way, we obtain a continuous homomorphism $\varphi : \Om mS
  \rightarrow \Om{n+1}S$ and we get $\varphi(s) \mathrel{\rho_{n+1}}
  \varphi(t)$. Now, we see that $0(\varphi(s))=v$ and
  $0(\varphi(t))=u$, which gives $v\mathrel{\rho^0_n} u$. Thus, we
  are done in this case.

  In the second case, we assume that all $\rho$-related pseudowords
  have the same first letter. In particular, we may assume that the
  first letter of $u$ and $v$ is $x_j$. However, there are $s$ and $t$
  such that $s\mathrel{\rho_m} t$ and $\cont(0(s))\not=\cont(0(t))$.
  Then, $\overline{0}(s)=y\not= z=\overline{0}(t)$ hold for some
  letters $y,z\in\{x_1,\dots ,x_m\}$. We substitute in both $s$ and
  $t$ the pseudowords $x_j^\omega vx_\ell$ for $z$, $x_j^\omega ux_k$
  for $y$ and $x_j^\omega$ for other variables. In this way, we obtain
  a continuous homomorphism $\varphi : \Om mS \rightarrow \Om{n+1}S$.
  Then, for the resulting $\rho$-related pseudowords $\varphi(s)$ and
  $\varphi(t)$, we have $0(\varphi(s))=x_j^\omega v$ and
  $0(\varphi(t))=x_j^\omega u$ by condition
  (\ref{e:rho-zero-outside-content}). Since $\sim_{\pv{CR}}\,
  \subseteq \rho$, $x_j^{\omega+1}$ is $\tilde{\rho}$-related to
  $x_j$. Hence, the initial factor $x_j^\omega$ in both $\varphi(s)$
  and $\varphi(t)$ can be removed because both $v$ and $u$ start with
  $x_j$. In this way, we obtain certain $\rho$-related pseudowords
  $s'$ and $t'$ such that $0(\varphi(s'))=v$ and $0(\varphi(t'))=u$
  which gives $v \mathrel{\rho^0} u$.\qed
\end{proof}

Now we are ready to prove the main result in this section.

\begin{Prop}
  \label{p:induction}
  Let $\pv H$ be a locally finite pseudovariety of groups. Let $\rho$
  be a complete system of pseudoinequalities such that
  ${\sim}_{\pv{CR(H)}}\subseteq \rho \subseteq {\sim}_{\pv{Sl}}$. Then
  $\rho$ is a complete system of pseudoidentities.
\end{Prop}

\begin{proof}
  We need to prove that, for every positive integer $n$, the relation
  $\rho_n$ is symmetric, i.e., we want to prove the implication
  \begin{equation}\label{e:induction-version-0}
    \forall u,v  : u\mathrel{\rho_n} v \implies v\mathrel{\rho_n} u \, .
  \end{equation}
  This implication is slightly informal as it is not clear how $n$ is
  quantified. At first, we need to clarify the relationship between
  the pseudowords $u,v$ and the index $n$. We know that $u\in \Om nS$
  implies $u\in \Om mS$ for every $m>n$. Therefore, for a pair of
  pseudowords $u\in \Om nS$ and $v\in \Om mS$, there is a minimum $k$
  such that $u,v\in\Om kS$, which we denote $n_{u,v}$ for the purpose
  of this proof. Thus, the index $n$ in the implication
  (\ref{e:induction-version-0}) is meant as this minimum index. Since
  $n$ depends on $u$ and $v$, we need to further clarify from which
  set these pseudowords are taken. So, we denote $\overline{\Omega}$
  the union of all $\Om nS$. Now we may write
  (\ref{e:induction-version-0}) in the following refined form:
  $$\forall u,v\in \overline{\Omega}  : 
  u\mathrel{\rho_{n_{u,v}}} v \implies v\mathrel{\rho_{n_{u,v}}} u.$$

  Moreover, we prove this statement for all possible $\rho$ satisfying
  the assumptions. For that purpose, we denote by $\Gamma_{\pv H}$ the
  class of all complete systems of pseudoinequalities $\rho$ such that
  ${\sim}_{\pv{CR(H)}}\subseteq \rho \subseteq {\sim}_{\pv{Sl}}$. Note
  that the assumption $\sim_{\pv{CR(H)}}\subseteq \rho$ immediately
  gives that $\rho_n$ has finite index for every~$n$. Now, we are
  ready to improve (\ref{e:induction-version-0}) into the final
  precise form:
  \begin{equation}\label{e:induction-version-final}
    \forall u,v \in\overline{\Omega},\  
    \forall \rho \in \Gamma_{\pv H} \  : 
    \ u\mathrel{\rho_{n_{u,v}}} v \implies 
    v\mathrel{\rho_{n_{u,v}}} u \, .
  \end{equation}
  We prove this statement by induction with respect to
  the size of $\cont(u)$.

  Let us assume that $|\cont(u)|=1$ and let $n=n_{u,v}$. Then, from
  the assumption that $\rho\subseteq {\sim}_{\pv{Sl}}$, we obtain the
  equalities $c(u)=c(v)=\{x_n\}$. Since \pv{CR(H)} is locally finite,
  there exist positive integers $a$ and $b$ such that
  $u\sim_{\pv{CR(H)}}x_n^a$ and $v\sim_{\pv{CR(H)}}x_n^b$. As
  ${\sim_{\pv{CR(H)}}}\subseteq\rho$, we deduce that
  $x_n^a\mathrel{\rho}x_n^b$. Raising to the power $\omega-1$ and
  multiplying by $x_n^{a+b}$, we obtain $x_n^b\mathrel{\rho}x_n^a$,
  whence $v\mathrel{\rho}u$.

  Assume next that $|\cont(u)|=k>1$ and that the statement
  (\ref{e:induction-version-final}) holds for every $u$ containing
  less than $k$ letters. Since $u \mathrel{\rho} v$, we get by
  Lemma~\ref{l:characterizing-rho} that
  \begin{equation}
    0(u)\mathrel{\rho^0} 0(v),%
    \ 1(u) \mathrel{\rho^1} 1(v),%
    \ u\sim_{\pv{Sl}}v,%
    \ uv^{\omega-1}\mathrel{\rho}(uv^{\omega-1})^\omega %
    \ \text{and} %
    \ vu^{\omega-1}\mathrel{\rho}(vu^{\omega-1})^\omega\, .
    \label{eq:induction-to-reverse}
  \end{equation}
  We want to use the same lemma which also gives the reverse
  implication. To prove that $v\mathrel{\rho}u$, we need to check that
  $v$ and $u$ may be interchanged in~\eqref{eq:induction-to-reverse}.
  To prove the implication $0(u)\mathrel{\rho^0} 0(v) \Rightarrow
  0(v)\mathrel{\rho^0} 0(u)$, we distinguish two cases. If
  $\rho^0\subseteq {\sim}_{\pv{Sl}}$ then we have
  $\rho^0=\rho^{0\cont} \in \Gamma_{\pv H}$ by
  Lemma~\ref{l:rho-0-c-is-ficsq}. We may now use the induction
  assumption, namely that the statement
  (\ref{e:induction-version-final}) is valid for the pair of
  pseudowords $0(u),0(v)\in\overline{\Omega}$. If $\rho^0$ is not
  contained in $\sim_{\pv{Sl}}$ then $\rho^0$ is symmetric by
  Lemma~\ref{l:case-rho-0-anti-content}. So, in both cases, we obtain
  $0(v)\mathrel{\rho^0} 0(u)$. The implication $1(u)\mathrel{\rho^1}
  1(v) \Rightarrow 1(v)\mathrel{\rho^1} 1(u)$ follows dually from all
  appropriate dual versions of the lemmas. Finally, we may conclude
  that $v \mathrel{\rho} u$ by Lemma~\ref{l:characterizing-rho}.\qed
\end{proof}

We are now ready to establish the main result, which is presented
as Theorem~\ref{t:main-above-sl} in the introduction.

\begin{Thm}
  \label{t:Sl}
  Let $\pv V$ be a pseudovariety of ordered completely regular
  semigroups containing all semilattices. Then $\pv V$ is selfdual.
\end{Thm}

\begin{proof}
  Let $(\pv{H}_i)_{i\ge1}$ be a sequence of locally finite
  pseudovarieties of groups such that $\pv{H}_i\subseteq \pv{H}_j$ for
  $i<j$ and $\bigcup_{i\ge1} \pv{H}_i=\pv G$. Consider, for every
  $i\ge1$, the pseudovariety of ordered completely regular semigroups
  $\pv{V}_i=\pv V \cap \pv{CR(H_i)}$. Clearly, we have
  $\pv{Sl}\subseteq \pv{V}_i \subseteq \pv{CR(H_i)}$. Next, consider
  the corresponding complete system of pseudoinequalities $\rho_{\pv
    V_i}$ for which we have ${\sim}_{\pv{CR(H_i)}}\subseteq \rho_{\pv
    V_i} \subseteq {\sim}_{\pv{Sl}}$. Then, by
  Proposition~\ref{p:induction}, $\rho_{\pv V_i}$ is a complete system
  of pseudoidentities. In other words, $\pv V_i$ is selfdual.

  Now, we have
  $$\bigcup_{i\ge1} \pv{V}_i %
  =\bigcup_{i\ge1} \left(\pv V \cap \pv{CR(H_i)}\right) %
  = \pv V \cap \bigcup_{i\ge1} \pv{CR(H_i)} %
  = \pv V \cap \pv{CR}=\pv V\, .$$ %
  Finally, the statement of the theorem follows from the observation
  that every directed union of selfdual pseudovarieties is a selfdual
  pseudovariety.\qed
\end{proof}

\section{Normal orthogroups}
\label{s:nocr}

As defined in the introduction, a normal orthogroup is a completely
regular semigroup for which the subsemigroup of idempotents is a
normal band. The pseudovariety of all finite normal orthogroups is
thus given by
$$\pv{NOCR}=\op %
x^{\omega+1}=x,\ %
x^\omega y^\omega = (x^\omega y^\omega)^\omega,\ %
x^\omega y^\omega x^\omega z^\omega x^\omega %
= x^\omega z^\omega x^\omega y^\omega x^\omega\cl .$$ %
A structural description of normal orthogroups is well known: they are
strong normal bands of groups, see, e.g.,
\cite[Section~IV.2]{Petrich&Reilly:1999}. Notice that one of the
consequences of this description is the equality
$\pv{NOCR}=\pv{NB}\vee\pv G$. In some sense, we improve the
characterization for the finite ordered case. Some of the following
results are well known and might be omitted (for example
Lemma~\ref{l:basic-properties-nocr}), however we keep most details to
make the paper self-contained. In comparison
with~\cite{Petrich&Reilly:1999}, we want to modify the results in
three directions. Firstly, we consider ordered semigroups and secondly
we work within the theory of pseudovarieties of finite semigroups. The
third direction, which is only implicitly contained
in~\cite[Theorem~IV.2.7]{Petrich&Reilly:1999}, is the relativization
with respect the fixed pseudovariety of groups.

Throughout this section, \pv V denotes a pseudovariety of ordered
semigroups. Before we introduce the results concerning ordered normal
orthogroups, we explain the importance of the pseudovariety
$\pv{NOCR}$ for our discussions.

\begin{Lemma}
  \label{l:not-above-semilattices}
  Let $\pv V\subseteq \pv{CR}$ be such that $\pv{Sl}\not\subseteq \pv
  V$. Then $\pv V\models x^\omega y^\omega x^\omega \le x^\omega$ or
  $\pv V\models x^\omega\le x^\omega y^\omega x^\omega$. In both cases
  $\pv V\subseteq \pv{NOCR}$.
\end{Lemma}

\begin{proof}
  We assume that $\rho_{\pv V} \not\subseteq {\sim}_{\pv{Sl}}$. This
  means that there is $(u,v) \in \rho_{\pv V}$ such that there is a
  variable $y$ that occurs in just one of the pseudowords $u$ and $v$.
  If we substitute $y^\omega$ for the variable $y$ and $x^\omega$ for
  all other variables and if we multiply the resulting
  pseudoinequality by $x^\omega$ from both sides, we obtain either
  $x^\omega y^\omega x^\omega\mathrel{\rho_{\pv{V}}} x^\omega$ or
  $x^\omega\mathrel{\rho_{\pv{V}}} x^\omega y^\omega x^\omega$. This
  gives the first part of the statement.

  Now assume that $x^\omega y^\omega x^\omega\mathrel{\rho_{\pv{V}}}
  x^\omega$ as the second case is dual. Multiplying by $y^\omega$
  yields the relation %
  $(x^\omega y^\omega)^2\mathrel{\rho_{\pv{V}}}x^\omega y^\omega$. %
  Multiplying by $x^\omega y^\omega$ and using transitivity, we deduce
  that %
  $(x^\omega y^\omega)^3\mathrel{\rho_{\pv{V}}}x^\omega y^\omega$. %
  Inductively, it follows that $%
  (x^\omega y^\omega)^{n!}\mathrel{\rho_{\pv{V}}}x^\omega y^\omega$ %
  for every $n\ge1$. The left side of this relation converges to
  $(x^\omega y^\omega)^\omega$. Since $\rho_{\pv V}$ is closed, we
  deduce that $(x^\omega y^\omega)^\omega\mathrel{\rho_{\pv
      V}}x^\omega y^\omega$. As $\rho_{\pv V}$ is a stable quasiorder
  and ${\sim_{\pv{CR}}}\subseteq\rho_{\pv V}$, we obtain $\pv
  V\models(x^\omega y^\omega)^\omega=x^\omega y^\omega$. Since every
  product of idempotents is idempotent modulo~\pv V, we get
  \begin{equation}
    \label{e:nocr}
    x^\omega y^\omega x^\omega z^\omega x^\omega \mathrel{\rho_{\pv V}}
    x^\omega y^\omega x^\omega z^\omega x^\omega\ \cdot \ 
    x^\omega y^\omega x^\omega z^\omega x^\omega\, .
  \end{equation}
  Using $x^\omega y^\omega x^\omega\mathrel{\rho_{\pv{V}}} x^\omega$
  and $x^\omega z^\omega x^\omega\mathrel{\rho_{\pv{V}}} x^\omega$ in
  the prefix and the suffix of the right hand side of~(\ref{e:nocr})
  we get $x^\omega y^\omega x^\omega z^\omega x^\omega
  \mathrel{\rho_{\pv V}} x^\omega z^\omega x^\omega y^\omega
  x^\omega$. Exchanging $y$ and $z$, we obtain $x^\omega z^\omega
  x^\omega y^\omega x^\omega \mathrel{\rho_{\pv V}} x^\omega y^\omega
  x^\omega z^\omega x^\omega$ which implies $\pv V\subseteq
  \pv{NOCR}$.\qed
\end{proof}

We recall some basic facts concerning normal orthogroups. We start by
recalling the solution of the word problem for normal bands which is
well known: for the pseudowords $u,v$, we have $\pv{NB} \models u=v$
if and only if $\cont(u)=\cont(v)$, the first letter in $u$ is the
same as the first letter in $v$ and the last letter in $u$ is the same
as the last letter in~$v$. This property implies, for example, that
$\pv{NOCR} \models x^\omega y^\omega x^\omega z^\omega x^\omega=
x^\omega y^\omega z^\omega x^\omega$. We use such pseudoidentities
freely. Other useful pseudoidentities are mentioned in the following
lemma.

\begin{Lemma}
  \label{l:basic-properties-nocr}
  The pseudovariety $\pv{NOCR}$ satisfies the pseudoidentities
  $(xy)^\omega=x^\omega y^\omega$ and $x^\omega y x^\omega z x^\omega
  = x^\omega y z x^\omega$.
\end{Lemma}

\begin{proof}
  Clearly, the pseudowords $x^\omega y^\omega$ and $xy$ are $\mathcal
  J$-related modulo $\sim_{\pv{NOCR}}$. Since we can interchange the
  placing of idempotents inside a product of idempotents we have
  $(xy)^\omega x^\omega y^\omega \sim_{\pv{NOCR}} x^\omega (xy)^\omega
  x^\omega y^\omega \sim_{\pv{NOCR}} x^\omega x^\omega (xy)^\omega
  y^\omega= (xy)^\omega$. In the same way we get $x^\omega y^\omega
  (xy)^\omega \sim_{\pv{NOCR}} (xy)^\omega$. Thus, modulo
  $\sim_{\pv{NOCR}}$, we have $(xy)^\omega \le_{\mathcal R} x^\omega
  y^\omega$ and also $(xy)^\omega \le_{\mathcal L} x^\omega y^\omega$.
  By stability of profinite semigroups, we deduce that $(xy)^\omega$
  and $x^\omega y^\omega$ are $\mathcal H$-related modulo
  $\sim_{\pv{NOCR}}$. Since both are idempotents, they are equal
  modulo $\sim_{\pv{NOCR}}$.
 
  We may derive the following sequence of pseudoidentities which are
  valid in $\pv{NOCR}$:
  \begin{align*}
    x^\omega y x^\omega z x^\omega
    &= (x^\omega y)^{\omega+1} x^\omega (z x^\omega)^{\omega+1}\\
    &= x^\omega y \cdot  
      x^\omega \cdot (x^\omega y)^{\omega} x^\omega (z x^\omega)^{\omega}
      \cdot x^\omega\cdot z x^\omega\\
    &= x^\omega y \cdot x^\omega \cdot (x^\omega y)^{\omega} (z
      x^\omega)^{\omega} \cdot x^\omega\cdot  z x^\omega\\
    &= (x^\omega y)^{\omega+1}  (z x^\omega)^{\omega+1}
      = x^\omega yzx^\omega.\ \qed
  \end{align*}
\end{proof}

The following lemma is a natural modification of the important
statement that, in every normal orthogroup, the $\mathcal H$-relation
is a congruence.

\begin{Lemma}
  \label{l:canonical-homomorphisms-on-band}
  Let $S$ be an ordered normal orthogroup. Then the mapping $\varphi
  :S \rightarrow E(S)$ given by the rule $\varphi(a)=a^\omega$ is a
  homomorphism of ordered semigroups. Moreover, if $\beta: S
  \rightarrow T$ is a surjective homomorphism such that $T$ is an
  ordered band, then $\beta$ can be factorized through $\varphi$.
\end{Lemma}

\begin{proof}
  The property $\varphi(st)=\varphi(s)\cdot \varphi(t)$, for every
  pair $s,t$ of elements of~$S$, follows from
  Lemma~\ref{l:basic-properties-nocr}. If $s\le t$ then $s^\omega\le
  t^\omega$ follows. Hence, $\varphi$ is also an isotone mapping.

  Now, for $s\in S$ we have $\beta(s)=\beta(s^2)$. The second part
  follows from the fact that $\beta\circ \varphi =\beta$, because
  $\beta(\varphi(s))=\beta(s^\omega)=\beta(s)$.\qed
\end{proof}

Before we exhibit another canonical surjective homomorphism from an
ordered normal orthogroup, we introduce some important examples of
ordered normal orthogroups. The first example of a normal orthogroup
is a $0$-group, that is, a group $G$ enriched by a zero element $0$.
We denote this semigroup $G^0$. Since we are interested in ordered
semigroups, we point out that the partial order that we consider on
$G^0$ is the equality. However, there are other possible stable
partial orders on this semigroup, which we denote in a different way.
The ordered semigroup $G^\top$ is the set $G\cup\{\top\}$ together
with the operation $\cdot$ which is an extension of the multiplication
on $G$ such that the element $\top$ is a zero element which is the
maximum element with respect to the partial order $\le$. Notice that
all the other elements are incomparable, because they are members of
the group. One can see that $G^\top$ is a normal orthogroup satisfying
the pseudoinequality $x^\omega \le x^\omega y^\omega x^\omega$. We
denote the dual ordered semigroup of $G^\top$ as $G^\perp$. This means
that the special element $\perp$ in $G^\perp$ is a zero with respect
to the multiplication and it is the minimum element with respect to
the partial order. Thus, $G^\perp$ satisfies the pseudoinequality
$x^\omega y^\omega x^\omega \le x^\omega$.

If we take $G$ to be a trivial group consisting of the idempotent
element~$1$, then we obtain in the previous construction an ordered
semilattice $\{1\}^\top=U^+=\{1,\top\}$ which satisfies the identity
$x\le xyx$. One can show that $\pv{Sl}^+ = \op x^2=x,xy=yx,x\le xy\cl
=\op x^2=x, x\le yxy\cl$ is the pseudovariety of ordered semigroups
generated by $U^+$: indeed, in every non-trivial ordered semilattice
$S$ satisfying $x\le xy$ we can choose two distinct elements $a,b$
such that $a\ne ab$ and consider the subsemigroup $\{a,ab\}$ which is
isomorphic to $U^+$. By the description, contained in
Proposition~\ref{p:emery} below, of the lattice of all pseudovarieties
of ordered normal bands by Emery~\cite{Emery:1999}, we also see that
$\pv{Sl}^+$ is a minimal pseudovariety of ordered semigroups. The
other minimal pseudovarieties in the lattice are the dual
pseudovariety $\pv{Sl}^{-}$, which is generated by the dual ordered
semigroup $U^{-}=\{1,\perp\}$, and the well known pseudovarieties of
left zero semigroups $\pv{LZ}=\op xy=x\cl$ and right zero semigroups
$\pv{RZ}=\op xy=y\cl$. The pseudovariety $\pv{LZ}$ is generated by a
two element semigroup $L=\{a,b\}$ with the multiplication given by the
rules $aa=ab=a$, $ba=bb=b$. Although one may order this semigroup by
$a<b$ to obtain an ordered semigroup $L^<$, this homomorphic image of
$L$ generates the same pseudovariety, because $L$ is isomorphic to the
subsemigroup of the product $L^< \times L^<$ consisting of two
incomparable elements $(a,b)$ and $(b,a)$. Finally, from left-right
duality we get that $\pv{RZ}$ is also generated by a single
two-element right zero semigroup $R$. Now, we are ready to recall the
description of the lattice of all pseudovarieties of normal
bands~\cite{Emery:1999}, including the subsequent reformulation using
minimal generators, which explains that the lattice is isomorphic to
the 4th power of a two-element lattice.

\begin{Prop}[\cite{Emery:1999}]
  \label{p:emery}
  The lattice of all pseudovarieties of ordered normal bands is
  presented on Figure~\ref{f:emery-lattice}. The characterization of
  every pseudovariety by an inequality is given inside $\pv B$, that
  is the identity $x^2=x$ is satisfied too.
  \begin{figure}
    \begin{center}
      \definecolor{lgray}{gray}{0.65}
      \definecolor{elgray}{gray}{0.85}
      \begin{tikzpicture}[x=0.7mm,y=1.05mm,thick]\small
        \node [] (1) at (0,0) {$\op x=y\cl$}; %
        \node [] (4) at (-14,10) {$\op x\leq yxy\cl$}; %
        \node [] (5) at (14,10) {$\op yxy\leq x\cl$}; %
        \node [] (6) at (0,20){$\op xy=yx\cl$}; %
        
        \node [] (2) at (-60,18) {$\op xy=x\cl$}; %
        \node [] (7) at (-74,28) {$\op x\leq xy\cl$}; %
        \node [] (8) at (-44,28) {$\op xy\leq x\cl$}; %
        \node [] (9) at (-60,38) {$\op xyz=xzy\cl$}; %
        
        \node [] (3) at (60,18) {$\op yx=x\cl$}; %
        \node [] (10) at (44,28) {$\op x\leq yx\cl$}; %
        \node [] (11) at (74,28) {$\op yx\leq x\cl$}; %
        \node [] (12) at (60,38) {$\op xyz=yxz\cl$}; %
        
        \node [] (13) at (0,36) {$\op x=xyx\cl$}; %
        \node [] (14) at (-14,46) {$\op x\leq xyx\cl$}; %
        \node [] (15) at (14,46) {$\op xyx\leq x\cl$}; %
        \node [] (16) at (0,56) {$\op xyzx=xzyx\cl$}; %
        
        \draw [color=elgray] (1) -- (2) -- (13) -- (3) -- (1); %
        \draw [color=elgray] (6) -- (9) -- (16) -- (12) -- (6); %
        \draw [color=elgray] (4) -- (7) -- (14) -- (10) -- (4); %
        \draw [color=elgray] (5) -- (8) -- (15) -- (11) -- (5); %
        \draw [color=lgray] (1) -- (4) -- (6) -- (5) -- (1); %
        \draw [color=lgray] (2) -- (7) -- (9) -- (8) -- (2); %
        \draw [color=lgray] (3) -- (10) -- (12) -- (11) -- (3); %
        \draw [color=lgray] (13) -- (14) -- (16) -- (15) -- (13); %
      \end{tikzpicture}
      \caption{The lattice of pseudovarieties of ordered normal bands.}
      \label{f:emery-lattice}
    \end{center}
  \end{figure}
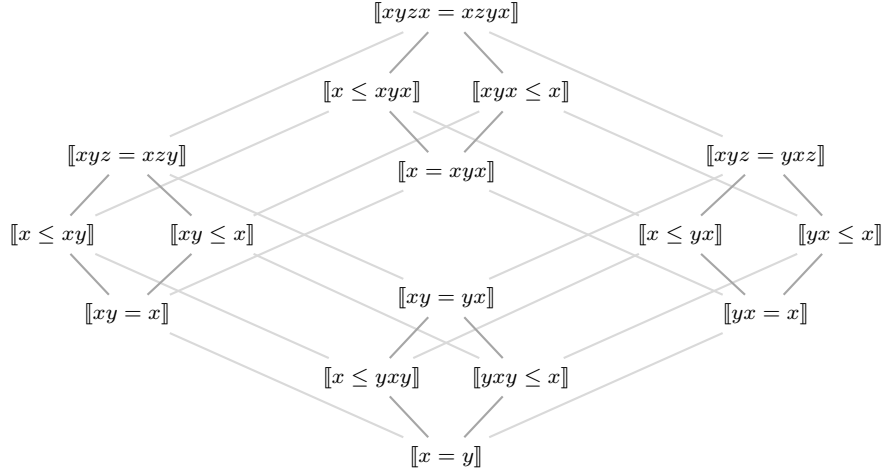
\end{Prop}

\begin{Cor}
  \label{c:L(ONB)}
  The lattice of all pseudovarieties of ordered normal bands is
  isomorphic to the set of all subsets of the set $\{U^+,U^-,L,R\}$
  partially ordered by the inclusion relation.
\end{Cor}

\begin{proof}
  A formal proof needs to check which inequalities are satisfied in
  the semigroups $U^+$, $U^-$, $L$ and $R$. We omit such technical
  computations and show just one example. The inequality $x\le xy$ is
  satisfied by $U^+$ and $L$ and it is not satisfied by $U^-$ and
  $R$.\qed
\end{proof}

Now, we return to our aim, namely to the description of all 
pseudovarieties of ordered normal orthogroups.
We start with one particular case.

\begin{Lemma}
  \label{l:rectangular-groups}
  Let $\pv V \subseteq \pv{NOCR}$ be such that $\pv V\models x^\omega =
  x^\omega y^\omega x^\omega$. Then, we have $\pv V = (\pv V \cap
  \pv{NB})\vee (\pv{V}\cap \pv G )$.
\end{Lemma}

\begin{proof}
  Let $S\in \pv{NOCR}$ be such that $S\models x^\omega = x^\omega
  y^\omega x^\omega$. This means that $E(S)$ is a rectangular band. In
  particular, $S$ is a completely simple semigroup. We choose one
  idempotent $e\in E(S)$ and consider $H_e$, the $\mathcal H$-class of
  $e$. Now, we consider the mapping $\psi : S \rightarrow H_e$ given
  by the rule $\psi(s)=ese$. Clearly, the mapping $\psi$ is isotone.
  Moreover, for every $s,t\in S$ we have $\psi(s) \cdot \psi(t) = ese
  \cdot ete =este=\psi(st)$ by Lemma~\ref{l:basic-properties-nocr}.
  Thus $\psi$ is a homomorphism of ordered semigroups.

  Next, we consider the mapping $\alpha : S \rightarrow E(S)
  \times H_e $ given by the rule
  $\alpha(s)=(\varphi(s),\psi(s))$, where $\varphi : S\rightarrow
  E(S)$ is described in Lemma~\ref{l:canonical-homomorphisms-on-band}.
  Since both $\varphi$ and $\psi$ are homomorphisms, $\alpha$ is a
  homomorphism as well. We claim that $\alpha$ is injective. Let $s$
  and $t$ be such that $\alpha(s)=\alpha(t)$. Then
  $s^\omega=\varphi(s)=\varphi(t)=t^\omega$ gives that $s$ and $t$
  belong to the same $\mathcal H$-class. By Green's Lemmas the mapping
  $s\mapsto ese$ is an injective mapping from $H_{s^\omega}$
  onto $H_e$. Hence, $s=t$ and the claim is proved.

  Since both mappings $\varphi$ and $\psi$ are surjective, it follows
  that $E(S)\in \pv V\cap \pv{NB}$ and $H_e\in \pv V\cap \pv
  G$. Thus, we obtain $S\in (\pv V \cap \pv{NB})\vee (\pv{V}\cap \pv G
  )$, which completes the proof.\qed
\end{proof}

Now, we extend the conclusion of the previous lemma to more pseudovarieties.

\begin{Lemma}
  \label{l:psv-is-join-plus-version}
  Let $\pv V \subseteq \pv{NOCR}$ be such that $\pv V\models
  x^\omega\le x^\omega y^\omega x^\omega$. Then $\pv V =(\pv V \cap
  \pv{NB})\vee (\pv{V}\cap \pv G ) $.
\end{Lemma}

\begin{proof}
  If $\pv V\models x^\omega = x^\omega y^\omega x^\omega$ then we can
  apply Lemma~\ref{l:rectangular-groups}. So, assume that there is a
  semigroup $T$ in $\pv V$ which does not satisfy the pseudoidentity
  $x^\omega = x^\omega y^\omega x^\omega$. Thus, there are $e,f\in
  E(T)$ such that $e\not = efe$, $e\le efe$. Hence, $U^+$ is
  isomorphic to the ordered subsemigroup $\{e,efe\}$ of $T$, and
  therefore $U^+$ belongs to $\pv V$.

  Now, let $S$ be an arbitrary ordered semigroup from $\pv V$. Let $e$
  be an arbitrary idempotent in $S$ and $H_e$ be the subgroup of $S$
  formed by the $\mathcal H$-class of~$e$. Recall that $\pv V\subseteq
  \pv{NOCR}$ implies $(ese)^\omega=es^\omega e$, for every $s\in S$.
  So, assuming $es^\omega e=e$, we obtain $ese\in H_e$. Hence we can
  define the mapping $\psi_e : S \rightarrow H_e^\top$ given by the
  rule
  $$\psi_e(x) = \left\{
    \begin{array}{ll} 
      exe, &\text{if } ex^\omega e=e \\ 
      \top, & \text{otherwise.} 
    \end{array}
  \right. $$ We claim that $\psi_e$ is a homomorphism from the ordered
  semigroup $S$ to the ordered normal orthogroup $H_e^\top$.

  Let $s,t\in S$ be arbitrary elements. Since $es^\omega e\cdot e
  t^\omega e = es^\omega t^\omega e=e(st)^\omega e$, we deduce that
  $e(st)^\omega e=e$ if and only if both equalities $es^\omega e=e$
  and $et^\omega e=e$ are valid. Therefore, the disjunction
  $\psi_e(s)=\top$ or $\psi_e(t)=\top$ is equivalent to
  $\psi_e(st)=\top$. And, furthermore, if $\psi_e(st)\not=\top$, i.e.,
  $\psi_e(st)=este$, then $\psi_e(s)=ese$ and $\psi_e(t)=ete$. In this
  case, we get $\psi_e(s)\cdot \psi_e(t)=ese\cdot ete = este$ by
  Lemma~\ref{l:basic-properties-nocr}. Thus, in every case
  $\psi_e(st)=\psi_e(s)\cdot \psi_e(t)$.

  Now, assume that $s,t\in S$ are such that $s\le t$. We need to show
  that $\psi_e(s)\le \psi_e(t)$, which is trivially satisfied if
  $\psi_e(t)=\top$. So, assume that $et^\omega e=e$ and
  $\psi_e(t)=ete$. From $s\le t$ we get $ese\le ete$. Thus, $es^\omega
  e=(ese)^\omega \le (ete)^\omega=et^\omega e=e$. Since $S$ satisfies
  the pseudoidentity $x^\omega\le x^\omega y^\omega x^\omega$, we also
  get $e\le es^\omega e$. Hence, $es^\omega e=e$ and we obtain
  $\psi_e(s)=ese$. Now, the inequality $\psi_e(s)\le \psi_e(t)$
  follows from $s\le t$. We proved the claim.

  Note that $\psi_e$ is surjective if and only if $e$ does not belong
  to the minimal $\mathcal J$-class of $S$. If $e$ belongs to the
  minimal $\mathcal J$-class, then $\psi_e(S)=H_e$. These observations
  are not used in what follows, because we use another argument which
  ensures that $H_e^\top \in \pv V$. We just notice that $H_e^\top$
  may be seen as a homomorphic image of the product $H_e \times U^+$
  using the homomorphism $\beta:H_e \times U^+ \rightarrow H_e^\top$
  mapping $(g,1)\mapsto g$, $(g,\top)\mapsto \top$ for each $g\in
  H_e$. Since we have $U^+ \in \pv V\cap \pv{NB}$ and $H_e \in \pv{V}
  \cap \pv{G}$, we see that $H_e^\top\in (\pv V \cap \pv{NB})\vee
  (\pv{V}\cap \pv G )$. In the rest of the proof, we show that $S$ may
  be reconstructed from these ordered normal orthogroups and from the
  ordered normal band $E(S)\in \pv{V} \cap \pv{NB}$. To prove this, we
  use also the homomorphism $\varphi : S\rightarrow E(S)$ described in
  Lemma~\ref{l:canonical-homomorphisms-on-band}.

  We consider the following homomorphism
  $$\alpha :S \rightarrow E(S) \times \prod_{e\in E(S)} H_e^\top ,$$
  where, for every $s\in S$, we put $\alpha(s)=(\varphi (s),
  (\psi_e(s))_{e\in E(S)})$. Since $\varphi$ and all $\psi_e$'s are
  homomorphisms of ordered semigroups, the mapping $\alpha$ is also a
  homomorphism. We show that $\alpha$ is injective. Let $s,t\in S$ be
  such that $\alpha(s)=\alpha(t)$. Then,
  $\varphi(s)=s^\omega=t^\omega=\varphi(t)$ is an idempotent from
  $E(S)$ which we denote by $f$. Then, we have also
  $\psi_f(s)=\psi_f(t)$, from which we deduce that $fsf=ftf$. Since
  $f=s^\omega=t^\omega$ we get $s=s^{\omega+1}=fsf=ftf=t^{\omega+1}=t$
  and we proved that $\alpha$ is injective.

  Since $H_e^\top \in (\pv V \cap \pv{NB})\vee (\pv{V}\cap \pv G )$,
  for every $e\in E(S)$, we get $S\in (\pv V \cap \pv{NB})\vee
  (\pv{V}\cap \pv G )$ as well. As $S$ was an arbitrary member of $\pv
  V$ we have thus proved that $\pv V\subseteq (\pv V \cap \pv{NB})\vee
  (\pv{V}\cap \pv G )$.\qed
\end{proof}

Let $\pv{RB}=\op xyx=x, x^2=x\cl$ be the pseudovariety of all finite
rectangular bands. We are now ready for establishing a more precise
version of Theorem~\ref{t:main-nocr}.

\begin{Thm}
  \label{t:nocr}
  The mapping $\iota : \mathcal{L}_o(\pv{NOCR}) \rightarrow
  \mathcal{L}_o(\pv{Sl})\times \mathcal{L}(\pv{RB})\times
  \mathcal{L}(\pv{G})$ given by the rule $\iota (\pv V) = (\pv{V}\cap
  \pv{Sl}, \pv V \cap \pv{RB}, \pv V \cap \pv{G})$ is a lattice
  isomorphism.
\end{Thm}

\begin{proof}
  From the structure of the lattice $\Cl L_o(\pv{NB})$, we know that
  the correspondence $\pv V\mapsto(\pv V\cap\pv{Sl},\pv V\cap\pv{RB})$
  defines a lattice isomorphism $\Cl L_o(\pv{NB})\to\Cl
  L_o(\pv{Sl})\times\Cl L(\pv{RB})$. We consider the mapping $\gamma :
  \mathcal{L}_o(\pv{NB})\times \mathcal{L}(\pv{G}) \rightarrow
  \mathcal{L}_o(\pv{NOCR})$ given by the rule $\gamma(\pv V, \pv H)
  =\pv V\vee \pv H$. Clearly, both mappings $\iota$ and $\gamma$ are
  isotone mappings between the ordered sets $\mathcal{L}_o(\pv{NOCR})$
  and $ \mathcal{L}_o(\pv{NB})\times \mathcal{L}(\pv{G})$. We show
  that they are mutually inverse mappings. This implies that they are
  isomorphisms of ordered sets and therefore also lattice
  isomorphisms.

  To show that $\gamma \circ \iota$ is the identity mapping, we need
  to prove that $(\pv V \cap \pv{NB})\vee (\pv{V}\cap \pv G )=\pv V$
  for every $\pv V \subseteq \pv{NOCR}$. This is true if $\pv V\models
  x^\omega\le x^\omega y^\omega x^\omega$, by
  Lemma~\ref{l:psv-is-join-plus-version}. Clearly, one may use the
  dual version of the lemma if $\pv V\models x^\omega y^\omega
  x^\omega \le x^\omega$. So, we may assume that $\pv V$ does not
  satisfy any of these two pseudoinequalities. Thus, by
  Lemma~\ref{l:not-above-semilattices} we know that $\pv{Sl} \subseteq
  \pv V$. From Theorem~\ref{t:Sl} we get that $\pv V$ is
  selfdual. This means that if we consider an arbitrary ordered
  semigroup $(S,\cdot,\le)\in\pv V$, then the pseudovariety $\pv V$
  contains also the ordered semigroup $(S,\cdot,=)$. Since
  $(S,\cdot,\le)$ is a homomorphic image of $(S,\cdot,=)$ we may deal
  only with the case of unordered semigroups.

  So, let $(S,\cdot,=)\in \pv V$ be arbitrary. Now, it is possible to
  modify the proof of Lemma \ref{l:psv-is-join-plus-version} in such a
  way that $H_e^\top$ is replaced by $H_e^0$. Since the partial order
  on $S$ is equality, the mapping $\psi_e$ is trivially isotone.
  Moreover, $\pv V$~contains \pv{Sl} and, therefore, $H_e^0 \in (\pv V
  \cap \pv{NB})\vee (\pv{V}\cap \pv G )$. Thus, we conclude that $S\in
  (\pv V \cap \pv{NB})\vee (\pv{V}\cap \pv G )$. We have proved the
  required equality $(\pv V \cap \pv{NB})\vee (\pv{V}\cap \pv G )=\pv
  V$.

  To show that $\iota\circ \gamma$ is the identity mapping, we need to
  show that $(\pv V \vee \pv H )\cap \pv{NB}=\pv V$ and $(\pv V \vee
  \pv H )\cap \pv G=\pv H$ for every pair $\pv V \subseteq \pv{NB}$
  and $\pv H \subseteq \pv{G}$.

  For the first equality, one direction, namely $(\pv V \vee \pv
  H)\cap \pv{NB}\supseteq \pv V$, is trivial. Let $S$ be an arbitrary
  member of $(\pv V \vee \pv H )\cap \pv{NB}$. This means that there
  are $S_1\in \pv V$, $S_2\in\pv H$ and there is an ordered
  subsemigroup $T$ of $S_1 \times S_2$ such that there is a surjective
  homomorphism of ordered semigroups $\beta : T \rightarrow S$. Since
  $S_1 \times S_2 \in \pv V \vee \pv H \subseteq \pv{NB}\vee \pv G
  \subseteq \pv{NOCR}$, by
  Lemma~\ref{l:canonical-homomorphisms-on-band} we know that the
  restriction of $\beta$ to $E(T)$ is also a surjective homomorphism
  of ordered semigroups. Hence, we may assume that $T$ is an ordered
  band. Since $S_2$ is a group, we see that $T$ is isomorphic to a
  subsemigroup of $S_1$. Hence, we have $T \in \pv V$ and,
  consequently, $S \in \pv V$ and the equality $(\pv V \vee \pv H)\cap
  \pv{NB}= \pv V$ is proved.

  From the well known fact that, for a surjective homomorphism
  $\beta:T\to G$, where $T$ is a finite semigroup and $G$ is a group,
  there is a subgroup $H$ of $T$ such that $\beta(H)=G$, it follows
  that the mapping $\pv V\mapsto \pv V\cap\pv G$ is a complete
  homomorphism $\mathcal{L}_o(\pv S)\to\mathcal{L}(\pv G)$. This was
  observed in \cite[Theorem~3.1]{Auinger&Hall&Reilly&Zhang:1995} for
  the unordered case, but the easy proof is the same for the ordered
  case. In particular, we have $(\pv V\vee\pv H)\cap\pv G=\pv H$ and
  the proof is complete.\qed
\end{proof}

\section{Applications}
\label{sec:applications}

We gather in this section several applications of our results. 

Selfduality of pseudovarieties of ordered completely regular
semigroups may now be characterized by very simple conditions.

\begin{Thm}
  \label{t:selfduality}
  Let \pv V be a pseudovariety of ordered completely regular
  semigroups. Then the following conditions are equivalent:
  \begin{enumerate}[(i)]
  \item\label{item:selfduality-1} the pseudovariety \pv V is selfdual;
  \item\label{item:selfduality-2} the intersection $\pv V\cap\pv{NB}$
    is selfdual;
  \item\label{item:selfduality-3} the intersection $\pv V\cap\pv{Sl}$
    is selfdual;
  \item\label{item:selfduality-4} either both or none of the ordered
    semilattices $U^+$ and $U^-$ belong to~\pv V.
  \end{enumerate}
\end{Thm}

\begin{proof}
  The implications
  (\ref{item:selfduality-1})\,$\Rightarrow$\,(\ref{item:selfduality-2})\,$\Rightarrow$\,(\ref{item:selfduality-3})
  follow from the fact that the intersection of selfdual
  pseudovarieties is also selfdual. On the other hand,
  (\ref{item:selfduality-3})\,$\Rightarrow$\,(\ref{item:selfduality-4})
  is an immediate consequence of Corollary~\ref{c:L(ONB)}.
  
  To prove that
  (\ref{item:selfduality-4})\,$\Rightarrow$\,(\ref{item:selfduality-1}),
  note first that, if both $U^+$ and $U^-$ belong to~\pv V, then \pv V
  contains~\pv{Sl}; hence, \pv V is selfdual by
  Theorem~\ref{t:Sl}. On the other hand, if neither $U^+$ nor
  $U^-$ belong to~\pv V, then \pv V is contained in~\pv{NOCR} by
  Lemma~\ref{l:not-above-semilattices} and $\pv V=(\pv
  V\cap\pv{Sl})\vee(\pv V\cap\pv{RB})\vee(\pv V\cap\pv G)$ by
  Theorem~\ref{t:nocr}; by Corollary~\ref{c:L(ONB)}, we conclude that
  $\pv V\cap\pv{Sl}$ is the trivial pseudovariety and, therefore, \pv
  V~is selfdual.\qed
\end{proof}

Given a pseudovariety of ordered semigroups \pv V, note that $\pv
V\vee \pv V^d$ is the least selfdual pseudovariety containing~\pv V.
We call it the selfdual closure of~\pv V and denote it~$\widetilde{\pv
V}$.

\begin{Prop}
  \label{p:selfdual-closure-description}
  The selfdual closure of a pseudovariety \pv V of ordered completely
  regular semigroups is the join $\pv V\vee\pv V'$ where $\pv V'$ is
  the selfdual closure of\/ $\pv V\cap\pv{Sl}$.
\end{Prop}

\begin{proof}
  If \pv V contains~\pv{Sl}, then $\pv V\vee\pv V'=\pv V$ is selfdual
  by Theorem~\ref{t:Sl}. Otherwise, by Theorem~\ref{t:nocr} we have a
  decomposition $\pv V=(\pv V\cap\pv{Sl})\vee(\pv
  V\cap\pv{RB})\vee(\pv V\cap\pv G)$, where the only term that may not
  be selfdual is $\pv V\cap\pv{Sl}$. Since the selfdual closure
  of~$\pv V\cap\pv{Sl}$ is precisely $\pv V'$, the result is now
  immediate.\qed
\end{proof}

Following the terminology of~\cite{Almeida&Klima:2015a}, we say that a
pseudovariety of semigroups is \emph{order primitive} if it is not
generated by a proper subpseudovariety of ordered semigroups;
equivalently, a pseudovariety is order primitive if it is not the
selfdual closure of a non-selfdual pseudovariety of order semigroups.
We also recall the Krohn-Rhodes complexity pseudovarieties
\cite{Krohn&Rhodes:1968}. Denote by \pv A the pseudovariety consisting
of all finite aperiodic semigroups, that is finite semigroups all of
whose subgroups are trivial. Note that $\pv{CR}\cap\pv A=\pv B$. Given
pseudovarieties of semigroups \pv V and \pv W, $\pv V*\pv W$ denotes
the pseudovariety of semigroups generated by all semidirect products
$S*T$ with $S\in\pv V$ and $T\in\pv W$. It is well known that this
defines an associative operation on~$\Cl L(\pv S)$. The complexity
pseudovarieties $\pv C_n$ are defined recursively by $\pv C_0=\pv A$
and $\pv C_{n+1}=\pv C_n*\pv G*\pv A$. From the Krohn-Rhodes
decomposition theorem it follows that the ascending chain $\pv C_n$
covers~\pv S. For much more on the Krohn-Rhodes complexity,
see~\cite{Rhodes&Steinberg:2009qt}.

\begin{Thm}
  \label{t:order-primitive}
  Let $\pv V\in\Cl L(\pv{CR})$. Then \pv V is not order primitive if
  and only if\/ $\pv V\subseteq\pv{NOCR}$ and $\pv V\cap\pv B$ lies in
  the interval $[\pv{Sl},\pv{NB}]$.
\end{Thm}

\begin{proof}
  Suppose first that $\pv V=\widetilde{\pv W}$ with $\pv V\ne\pv W$.
  In particular, \pv W is not selfdual and so, by Theorem~\ref{t:Sl}
  and Lemma~\ref{l:not-above-semilattices}, we must have $\pv
  V\subseteq\pv{NOCR}$, so that $\pv V\cap\pv B\subseteq\pv{NB}$. If
  \pv{Sl} is not contained in~\pv V, then neither $U^+$ nor $U^-$ can
  belong to~\pv W; by Theorem~\ref{t:selfduality}, it then follows
  that \pv W must be selfdual, which contradicts our assumptions.
  Hence, the inclusion $\pv{Sl}\subseteq\pv V$ holds.

  For the converse, consider the pseudovariety $\pv U=\pv V\cap\pv B$.
  By inspection of Figure~\ref{f:emery-lattice}, the assumption that
  \pv U belongs to the interval $[\pv{Sl},\pv{NB}]$ implies that $\pv
  U=\widetilde{\pv X}$, where $\pv X=(\pv
  U\cap\pv{RB})\vee\pv{Sl}^+$. From Theorem~\ref{t:nocr}, it follows
  that $\pv V=\widetilde{\pv W}$, where $\pv W=\pv X\vee(\pv V\cap\pv
  G)$. Moreover, \pv W is not selfdual by Theorems~\ref{t:nocr}
  and~\ref{t:selfduality}. Hence, \pv V is not order primitive.\qed
\end{proof}

The following result settles one of the problems left open
in~\cite[Table~3]{Almeida&Klima:2015a}.
 
\begin{Cor}
  \label{c:CRcapCn}
  For every $n\ge0$, the pseudovariety $\pv{CR}\cap\pv C_n$ is order
  primitive.
\end{Cor}

\begin{proof}
  It suffices to note that $\pv C_n$ contains~\pv B, so that
  $\pv{CR}\cap\pv C_n$ fails both criteria of
  Theorem~\ref{t:order-primitive}.\qed
\end{proof}

The following extends to the ordered case a result of
Pastijn~\cite{Pastijn:1991}.

\begin{Thm}
  \label{t:modular}
  The lattice $\Cl L_o(\pv{CR})$ is modular.
\end{Thm}

\begin{proof} We must show that the pentagon
  
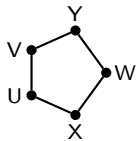
\begin{center}
    \begin{tikzpicture}[x=1.5mm,y=2mm,thick]\small %
      \tikzstyle{every node}=[circle, draw, fill=black, inner sep=0pt,
                              minimum width=3pt] %
      \node [label={left: \pv U}] (U) at (-3.85,4.5) {}; %
      \node [label={left: \pv V}] (V) at (-3.85,7.5) {}; %
      \node [label={right: \pv W}] (W) at (2.7,6) {}; %
      \node [label={below: \pv X}] (0) at (0,3.2) {}; %
      \node [label={above: \pv Y}] (1) at (0,8.8) {}; %
      \draw (U) -- (V) -- (1) -- (W) -- (0) -- (U); %
    \end{tikzpicture}
  \end{center}
  does not appear as a sublattice of~$\Cl L_o(\pv{CR})$. Suppose, on
  the contrary, that it does appear. The idea of the proof is to show
  that it is possible to render selfdual all vertices of the pentagon
  still retaining a pentagon, which leads to a contradiction since the
  lattice $\Cl L(\pv{CR})$ is modular
  \cite[Corollary~8]{Pastijn:1991}.

  We first observe that \pv Y cannot belong to~$\Cl L_o(\pv{NOCR})$
  since this lattice is modular as it is isomorphic to a product of
  modular lattices by Theorem~\ref{t:nocr}. Note also that either at
  least one of~\pv U and~\pv V is not selfdual or~\pv W is not
  selfdual, but not both conditions can hold: otherwise, either all
  vertices of the pentagon are selfdual or they all belong to~$\Cl
  L_o(\pv{NOCR})$. Moreover, the (left: \pv U, \pv V; or right: \pv W)
  side of the pentagon where non-selfduality is not present must
  contain~\pv{Sl} while, on the other side, none of the vertices
  contains~\pv{Sl}: first, if on both sides we would have
  pseudovarieties not containing \pv{Sl}, then the pentagon would be
  found within~$\Cl L_o(\pv{NOCR})$ by
  Lemma~\ref{l:not-above-semilattices}; second, if \pv{Sl} is
  contained in at least one pseudovariety on each side, then it is
  contained in~\pv X and the pentagon would be placed in~$\Cl
  L(\pv{CR})$ by Theorem~\ref{t:Sl}. The same argument shows that, \pv
  U is selfdual if and only if so is \pv V. In any case, \pv X cannot
  be selfdual.

  Let $\pv{RG}=\pv{RB}\vee\pv G$, the pseudovariety of all finite
  so-called \emph{rectangular groups}. We replace in the pentagon each
  non-selfdual vertex by its intersection with~\pv{RG}. We show that
  this produces a sublattice of~$\Cl L(\pv{CR})$ which is still a
  pentagon.

  Consider first the case where \pv U, \pv V, and \pv X are not
  selfdual. Then, these three pseudovarieties are replaced by their
  intersections with~\pv{RG}. Doing so, the new bottom pseudovariety
  is still the only intersection of the two new sides. On the other
  hand, as \pv W contains~\pv{Sl}, each of \pv U and \pv V is
  contained in the join with~\pv W of its intersection with~\pv{RG};
  hence, \pv Y remains the only join of the new two sides. Thus, we
  obtain a pentagon, except if there is some side which collapses. The
  only possible collapse would come from the identification of two
  vertices being intersected with~\pv{RG}. But, since \pv U, \pv V,
  and \pv X are not selfdual and form a chain, their intersections
  with~\pv{Sl} must be the same. Hence, by Theorem~\ref{t:nocr}, their
  intersections with~\pv{RG} must remain distinct.

  Finally, consider the case where \pv W is not selfdual. Then, both
  \pv W and \pv X are replaced by their intersections with~\pv{RG}.
  The argument in the preceding paragraph then allows us to show that
  the modified pentagon is still a pentagon. Thus, in all cases, we
  reach the announced contradiction, thereby proving the theorem.\qed
\end{proof}

We next recall the following result of Reilly and Zhang.

\begin{Thm}[\cite{Reilly&Zhang:1997}]
  \label{t:RZ}
  The correspondence $\pv V\mapsto\pv V\cap\pv B$ defines a complete
  endomorphism of the lattice $\Cl L(\pv S)$.
\end{Thm}

As a step in the proof that intersection with \pv B remains a complete
endomorphism of the lattice $\Cl L_o(\pv S)$, we first show that that
is the case for the restriction to~$\Cl L_o(\pv{CR})$.

Since $\pv{NOCR}=\pv{NB}\vee\pv G$ and taking the intersection with
\pv B determines a complete endomorphism of~$\Cl L(\pv S)$, it follows
that $\pv V\cap\pv B=\pv V\cap\pv{NB}$ for every $\pv V\in\Cl
L_o(\pv{NOCR})$. We use this property freely for the rest of the
paper.

\begin{Lemma}
  \label{l:capB-vs-selfdual-closure}
  Let \pv V be a pseudovariety of ordered completely regular
  semigroups. Then, we have $\widetilde{\pv V\cap\pv B}=\widetilde{\pv
    V}\cap\pv B$.
\end{Lemma}

\begin{proof}
  If \pv V is selfdual, then so is $\pv V\cap\pv B$ and this is the
  value of both sides of the equality in the statement of the lemma.
  If \pv V is not selfdual then $\widetilde{\pv V}$ is contained
  in~$\pv{NOCR}$. By Proposition~\ref{p:selfdual-closure-description},
  since \pv V is not selfdual, we have $\widetilde{\pv V}=\pv
  V\vee\pv{Sl}$. We thus obtain the following chain of equalities:
  $$\widetilde{\pv V\cap\pv B} %
  =\widetilde{\pv V\cap\pv{NB}} %
  =(\pv V\cap\pv{NB})\vee\pv{Sl} %
  =(\pv V\vee\pv{Sl})\cap\pv{NB} %
  =\widetilde{\pv V}\cap\pv{NB} %
  =\widetilde{\pv V}\cap\pv B.\ \qed $$
\end{proof}

\begin{Thm}
  \label{t:capB}
  The correspondence $\pv V\mapsto\pv V\cap\pv B$ defines a complete
  endomorphism of the lattice $\Cl L_o(\pv{CR})$.
\end{Thm}

\begin{proof}
  Obviously, taking intersection with \pv B preserves arbitrary
  intersections. Consider a family $(\pv V_i)_{i\in I}$ of members
  of~$\Cl L_o(\pv{CR})$.
  If none of the $\pv V_i$ contains~\pv{Sl}, then they all belong
  to~$\Cl L_o(\pv{NOCR})$ by~Lemma~\ref{l:not-above-semilattices};
  hence, in this case, the result follows from the fact that
  intersection with~\pv{NB} defines a complete endomorphism of~$\Cl
  L_o(\pv{NOCR})$ by Theorem~\ref{t:nocr}. It remains to consider the
  case where at least one of the $\pv V_i$ contains~\pv{Sl}. In such a
  case, by Proposition~\ref{p:selfdual-closure-description}, we have %
  $\bigvee_{i\in I}\pv V_i %
  =\bigvee_{i\in I}\widetilde{\pv V_i}$ and %
  $\bigvee_{i\in I}(\pv V_i\cap\pv B) %
  =\bigvee_{i\in I}\widetilde{\pv V_i\cap\pv B}$. %
  On the other hand, by Lemma~\ref{l:capB-vs-selfdual-closure} and
  Theorem~\ref{t:RZ} we obtain the equalities %
  $\bigvee_{i\in I}\widetilde{\pv V_i\cap\pv B} %
  =\bigvee_{i\in I}(\widetilde{\pv V_i}\cap\pv B) %
  =(\bigvee_{i\in I}\widetilde{\pv V_i})\cap\pv B $. %
  Combining the two preceding observations, we obtain the desired
  equality $\bigvee_{i\in I}(\pv V_i\cap\pv B)=(\bigvee_{i\in I}\pv
  V_i)\cap\pv B$.\qed
\end{proof}

Next, we use the method of Reilly and Zhang~\cite{Reilly&Zhang:1997}
to produce complete endomorphisms of the lattice $\Cl L_o(\pv S)$.
Following the terminology in that paper, we denote by \pv{DCh} the
class of finite completely regular semigroups $S$ such that the
partially ordered set $S/\Cl J$ is a chain. The next two lemmas are
borrowed from~\cite{Reilly&Zhang:1997}. The first is the core of the
argument.

\begin{Lemma}[{\cite[Lemma~3.10]{Reilly&Zhang:1997}}]
  \label{l:RZ-3.10}
  Let $S\in\pv{DCh}$, $T\in\pv S$, and $\varphi:T\to S$ be a
  surjective homomorphism. Then there exists a subsemigroup $R$ of~$T$
  such that $R\in\pv{DCh}$ and the restriction $\varphi|_R$ is
  surjective.
\end{Lemma}

For a class \Cl C of finite ordered semigroups, denote by $\langle\Cl
C\rangle_o$ the pseudovariety of ordered semigroups generated by~\Cl
C.

\begin{Lemma}[see {\cite[Lemma~3.11]{Reilly&Zhang:1997}}]
  \label{l:RZ-3.11}
  Let $\pv U,\pv V\in\Cl L_o(\pv S)$. If $S\in\pv{DCh}$ is such
  that $S\in\pv U\vee\pv V$, then $S\in\langle\pv
  U\cap\pv{DCh}\rangle_o\vee\langle\pv V\cap\pv{DCh}\rangle_o$.
\end{Lemma}

\begin{proof}
  The proof of \cite[Lemma~3.11]{Reilly&Zhang:1997} carries through
  unchanged to the ordered case; it is a simple application of
  Lemma~\ref{l:RZ-3.10}.\qed
\end{proof}

We may now state the ordered version of the main result
of~\cite[Theorem~3.12]{Reilly&Zhang:1997}.

\begin{Thm}
  \label{t:RZ-3.12}
  Let $\pv W\in\Cl L_o(\pv{CR})$ have the following properties:
  \begin{enumerate}[(i)]
  \item\label{item:RZ-3.12-1} the correspondence $\pv V\mapsto\pv
    V\cap\pv W$ defines an endomorphism of the lattice $\Cl
    L_o(\pv{CR})$;
  \item\label{item:RZ-3.12-2} if\/ $\pv V\in\Cl L_o(\pv W)$ then we
    have $\pv V=\langle\pv V\cap\pv{DCh}\rangle_o$.
  \end{enumerate}
  Then the correspondence $\pv V\mapsto\pv V\cap\pv W$ defines a complete
  endomorphism of the lattice $\Cl L_o(\pv S)$.
\end{Thm}

\begin{proof}
  Taking into account Lemma~\ref{l:RZ-3.11}, the proof
  of~\cite[Theorem~3.12]{Reilly&Zhang:1997} carries over unchanged to
  the ordered case to show that the correspondence $\pv V\mapsto\pv
  V\cap\pv W$ defines an endomorphism of the lattice $\Cl L_o(\pv S)$.
  That it is actually a complete endomorphism follows
  from the order analog of~\cite[Lemma~3.2]{Reilly&Zhang:1997}, whose
  proof requires no essential changes.\qed
\end{proof}

The following in an immediate application of Theorems~\ref{t:capB}
and~\ref{t:RZ-3.12}.

\begin{Cor}
  \label{c:capB}
  The correspondence $\pv V\mapsto\pv V\cap\pv B$ defines a complete
  endomorphism of the lattice $\Cl L_o(\pv S)$.
\end{Cor}
 
\begin{proof}
  It only remains to verify the hypothesis (\ref{item:RZ-3.12-2}) of
  Theorem~\ref{t:RZ-3.12}, namely that $\pv V=\langle\pv
  V\cap\pv{DCh}\rangle_o$ for every~$\pv V\in\Cl L_o(\pv B)$. This is
  well known in case \pv V is selfdual. For each of the eight
  non-selfdual $\pv V\in\Cl L_o(\pv B)$, which appear in
  Figure~\ref{f:emery-lattice}, we gave in Section~\ref{s:nocr}
  generating sets consisting of elements of~\pv{DCh}, namely a
  suitable subset of the set $\{U^+,U^-,L,R\}$.\qed
\end{proof}

Adopting the terminology
of~\cite[Definition~6.1.5]{Rhodes&Steinberg:2009qt}, we say that an
element $a$ of a lattice is strictly finite join irreducible (sfji) if
$a=b\vee c$ implies $a=b$ or $a=c$; and we say that $a$ is finite join
irreducible (fji) if $a\le b\vee c$ implies $a\le b$ or $a\le c$.

The following result solves a problem left open
in~\cite[Table~3]{Almeida&Klima:2015a}.

\begin{Cor}
  \label{c:B-fji}
  The pseudovariety \pv B is fji in the lattice $\Cl L_o(\pv S)$.
\end{Cor}

\begin{proof}
  It is easy to see that, if \pv V is sfji and intersection with \pv V
  distributes over finite joins, then \pv V is fji. On the other hand,
  it follows from Ku\v ril's characterization of~$\Cl L_o(\pv B)$
  \cite{Kuril:2015} that \pv B is sfji in $\Cl L_o(\pv S)$.\qed
\end{proof}

\section{Final remarks}
\label{sec:final-remarks}

Although we have dealt in this paper mostly with (pro)finite
semigroups, Theorems~\ref{t:main-above-sl}, \ref{t:main-nocr},
and~\ref{t:main-modular} extend with somewhat simpler proofs to the
case of varieties of ordered completely regular semigroups. In
particular, there is no need to take care of reducing to the locally
finite case in the proof of the variety analog of
Theorem~\ref{t:main-above-sl}.

It would of interest to extend our results to the non-regular analog
of finite completely regular semigroups, namely the finite semigroups
in which every regular element lies in a group. They form a
pseudovariety which is commonly denoted \pv{DS} in the literature and
that was first considered by Sch\"utzenberger
\cite{Schutzenberger:1976} in connection with formal language theory.
By Green's Lemmas, \pv{DS} may alternatively be characterized as
consisting of all finite semigroups in which regular \Cl J-classes are
subsemigroups. The class \pv{DO} of all finite semigroups whose
regular \Cl J-classes form orthodox subsemigroups may be easier to
handle as it is much better understood \cite{Almeida:1996c}. Because
of connections with logic and formal language theory
\cite{Tesson&Therien:2002,Kufleitner&Weil:2009}, it is of particular
interest to consider the pseudovariety \pv{DA}, which consists of all
finite semigroups whose regular elements are idempotents. It is shown
in~\cite[Table~3]{Almeida&Klima:2015a} that $\pv{DS}\cap\overline{\pv
  H}$ is fji in $\Cl L_o(\pv S)$, where \pv H is an arbitrary
nontrivial pseudovariety of groups; in contrast, it is not known
whether the same holds for $\pv{DO}\cap\overline{\pv H}$ and
for~\pv{DA}. The lack of torsion makes the approach
of~\cite{Almeida&Klima:2015a} fail for subpseudovarieties of~\pv{DO}
but may facilitate the extension of the methods in the present paper.

\bibliographystyle{spmpsci}
\bibliography{OrderedCR-arxiv}

\end{document}